\documentclass{amsart}
\usepackage{verbatim}
\usepackage{amssymb}
\usepackage{amsmath}
\usepackage[usenames]{color}
\usepackage{graphicx}
\usepackage{amscd}
\usepackage[numbers,sort&compress,square]{natbib}

\theoremstyle{plain}
\newtheorem{theorem}{Theorem}
\newtheorem{prop}[theorem]{Proposition}

\newtheorem{lemma}[theorem]{Lemma}

\newtheorem{defi}[theorem]{Definition}
\newtheorem{conj}[theorem]{Conjecture}

\theoremstyle{definition}

\newtheorem{remark}[theorem]{Remark}

\newtheorem{example}[theorem]{Example}
\numberwithin{equation}{section}
\numberwithin{theorem}{section}
\allowdisplaybreaks

\renewcommand{\b}{\beta}

\newcommand{\eps}{\varepsilon}

\newcommand{\sign}{\mathrm{sign}\text{ }}

\newcommand{\dd}[0]{\mathrm{d}}
\newcommand{\ud}[0]{\,\mathrm{d}}

\newcommand{\black}{\color{black}}

\newcommand{\vertiii}[1]{{\left\vert\kern-0.25ex\left\vert\kern-0.25ex\left\vert #1 
    \right\vert\kern-0.25ex\right\vert\kern-0.25ex\right\vert}}

\begin{document}

\title[UMD spaces and weak differential subordination]
{Fourier multipliers and weak differential subordination 
 of martingales  in UMD Banach spaces}

\author{Ivan Yaroslavtsev}
\address{Delft Institute of Applied Mathematics\\
Delft University of Technology \\ P.O. Box 5031\\ 2600 GA Delft\\The
Netherlands}
\email{I.S.Yaroslavtsev@tudelft.nl}

\begin{abstract}
In this paper we introduce the notion of weak differential subordination for 
martingales and  show that a Banach space $X$ is a UMD Banach space if and 
only if for all $p\in 
(1,\infty)$ and all purely discontinuous $X$-valued martingales $M$ and $N$ 
such that $N$ is weakly differentially subordinated to $M$, one has the estimate 
$\mathbb E \|N_{\infty}\|^p \leq C_p\mathbb E \|M_{\infty}\|^p$. As a corollary 
we derive the sharp estimate for the norms of a broad class of even Fourier 
multipliers, which includes e.g.\ the second order Riesz transforms. 
\end{abstract}



\maketitle

\section{Introduction}

Applying stochastic techniques to Fourier multiplier theory has a long history 
(see e.g.\ \cite{BBB,BB,BanWang95,Bour83,Burk83,Gar85,GM-SS,Mcc84}). It 
 turns out that 
the boundedness of certain Fourier multipliers  with values in a Banach space $X$
is equivalent to this Banach space being in a special class, namely 
in the class of UMD Banach spaces. 
Burkholder in \cite{Burk83} and Bourgain in \cite{Bour83} showed that the 
Hilbert transform is bounded on $L^p(\mathbb R; X)$ for $p\in (1,\infty)$ if and 
only if $X$ is UMD. The same type of assertion can be proven for the 
Beurling-Ahlfors transform, see the paper \cite{GM-SS} by Geiss, Montgomery-Smith and 
Saksman. Examples of UMD spaces include the reflexive range of $L^q$-, Sobolev 
and Besov spaces. 

 A more general class of Fourier multiplier has been considered in recent 
works of Ba\~{n}uelos and Bogdan \cite{BB} and Ba\~{n}uelos, Bielaszewski and Bogdan \cite{BBB}.
 They derive sharp estimates for the norm of a Fourier multiplier with symbol
\begin{equation}\label{eq:introsymbol}
 m(\xi) = \frac{\int_{\mathbb R^d} (1-\cos \xi \cdot z)\phi(z)V(\dd z) + 
\frac 12\int_{S^{d-1}}(\xi \cdot \theta)^{2}\psi(\theta)\mu(\dd\theta)}{\int_{\mathbb 
R^d} (1-\cos \xi \cdot z)V(\dd z)+ \frac 12\int_{S^{d-1}}(\xi \cdot 
\theta)^{2}\mu(\dd\theta)},\;\; \xi \in \mathbb R^d,
\end{equation}
on $L^p(\mathbb R^d)$.
Here we will extend their result to $L^p(\mathbb R^d; X)$ for UMD spaces $X$. 
 More precisely, we will show that a Fourier multiplier $T_m$ with a symbol of the form
\eqref{eq:introsymbol} is bounded on $L^p(\mathbb R^d; X)$ if $V$ is a L\'evy 
measure, $\mu$ is a Borel positive measure, $|\phi|, |\psi|\leq 1$, and that 
then  the norm of $T_m$
does not exceed the UMD$_p$ constant of~$X$. In Subsection 
\ref{sec:examples}, several examples of symbols $m$ of the form \eqref{eq:introsymbol}
are given, and we will see that for 
some particular symbols $m$ the norm of $T_m$ equals the UMD constant.

\medskip

To prove the generalization of the results in \cite{BBB,BB} we will need 
additional geometric properties of a UMD Banach space. In the 
fundamental paper \cite{Burk86}, Burkholder showed that a Banach space $X$ is UMD 
if and only if for some $\beta>0$ there exists a zigzag-concave function 
$U:X\times X \to \mathbb R$ (i.e.,\ a function $U$ such that $U(x+z, y+\eps z)$ 
is concave in $z$ for any sign $\eps$ and for any $x, y\in X$) such that $U(x, 
y) \geq \|y\|^p - \beta^p\|x\|^p$ for all $x,y \in X$. Such a function $U$ is 
called a~{\it Burkholder function}. 
 In this situation, we can in fact take $\beta$  equal to the 
UMD$_p$  constant of $X$
(see Subsection~\ref{subsec:UMD} and Theorem \ref{thm:Burkholder}).
By exploiting appropriate Burkholder functions $U$ 
one can prove a wide variety of interesting results
(see \cite{BanJan08,BanOs13,BanWang95,BanWang96,Burk84,Burk85a,Wang} and the works 
\cite{Os11a,Os12a,Os12b,Os12c,Os13a,Os08a,Os08b,Os15a,Os12,Os07a} by Os{\c{e}}kowski). 
For our purposes the 
following result due to Burkholder \cite{Burk84} (for the scalar case) and Wang \cite{Wang} (for the 
Hilbert space case) is of special importance: 
\begin{theorem}\label{thm:introBurWan}
 Let $H$ be a Hilbert space, $(d_n)_{n\geq 0}$, $(e_n)_{n\geq 0}$ be two 
$H$-valued martingale difference sequences such that $\|e_n\|\leq \|d_n\|$ a.s.\ 
for all $n\geq 0$. Then for each $p\in (1,\infty)$,
 \[
  \mathbb E \Bigl\|\sum_{n\geq 0} e_n\Bigr\|^p\leq (p^*-1)^p \mathbb 
E\Bigl\|\sum_{n\geq 0} d_n\Bigr\|^p.
 \]
\end{theorem}
Here and in the sequel $p^* = \max(p, p')$, where $\frac 1p + 
\frac 1{p'}=1$. This result cannot be generalized beyond the 
Hilbertian setting; see \cite[Theorem~3.24(i)]{Os12} and \cite[Example 
4.5.17]{HNVW1}. 
In the present 
paper we will show the following UMD variant of Theorem 
\ref{thm:introBurWan}:
\begin{theorem}\label{thm:introweakdiffsub}
 Let $X$ be a UMD space, $(d_n)_{n\geq 0}$, $(e_n)_{n\geq 0}$ be two $X$-valued 
martingale difference sequences, $(a_n)_{n\geq 0}$ be a scalar-valued adapted 
sequence such that $|a_n|\leq 1$ and $e_n = a_n d_n$ for all $n\geq 0$. Then for 
each $p\in (1,\infty)$
 \[
  \mathbb E \Bigl\|\sum_{n\geq 0} e_n\Bigr\|^p\leq \beta^p_{p,X} \mathbb 
E\Bigl\|\sum_{n\geq 0} d_n\Bigr\|^p,
 \]
\end{theorem}
\noindent where $\beta_{p,X}$ is the UMD$_p$-constant of $X$ (notice that Burkholder proved the identity $\beta_{p,H} = p^*-1$ for a Hilbert space $H$, see \cite{Burk84}). 
Theorem \ref{thm:introweakdiffsub} generalizes 
a famous Burkholder's result \cite[Theorem 2.2]{Burk81} on martingale 
transforms, where $(a_n)_{n\geq 0}$ was supposed to be predictable.
The main tool for proving Theorem \ref{thm:introweakdiffsub} is a Burkholder 
function with a stricter zigzag-concavity: now we also require $U(x+z, y+\eps 
z)$ to be concave in $z$ for any $\eps$ such that $|\eps|\leq 1$. In the 
finite dimensional case one gets it for free thanks to the existence of an explicit formula of $U$ (see Remark \ref{rem:C_p} and \cite{Wang}). Here we show the existence of such 
a Burkholder function in infinite dimension.

\smallskip

For the applications of our abstract results to the theory of Fourier multipliers 
we extend Theorem 
\ref{thm:introweakdiffsub} to the continuous time setting. Namely, we show an 
analogue of Theorem \ref{thm:introweakdiffsub} for purely discontinuous martingales (i.e.\ martingales which quadratic variations are pure jump processes, see Subsection \ref{subsec:contcase}). 

\smallskip

An extension of Theorem \ref{thm:introweakdiffsub} to general continuous-time 
martingales is shown in the 
paper \cite{Y17MartDec}. 
Nevertheless, the sharp estimate in this extension for the case of continuous martingales remains an open problem. This problem 
is in fact of interest in Harmonic Analysis. If true, this sharp estimate 
can be used to 
study a larger class of multipliers, including the Hilbert transform~$\mathcal 
H_X$. Garling in \cite{Gar85} proved that $$ \|\mathcal H_X\|_{\mathcal 
L(L^p(\mathbb R; X))}\leq \beta^2_{p,X},$$
 and it is a long-standing open problem (see \cite[pp.496--497]{HNVW1})
to prove a linear estimate of the form 
$$\|\mathcal H_{X}\|_{\mathcal 
L(L^p(\mathbb R; X))} \leq C\beta_{p,X}$$ for some constant~$C$.
Here we will show that the latter estimate would indeed follow if
one can show the existence of a Burkholder function with certain additional properties. 
At present, the existence of such Burkholder functions is known
only in the Hilbert space case (see Remark \ref{rem:C_p}).

\medskip

\emph{Acknowledgment} -- The author would like to thank Mark Veraar for 
inspiring discussions and helpful comments, and Jan van Neerven and Emiel Lorist 
for careful reading of parts of this paper and useful suggestions. The author thanks the anonymous referee for his/her valuable comments.

\section{Preliminaries}
\subsection{UMD Banach spaces}\label{subsec:UMD}\nopagebreak
A Banach space $X$ is called a {\it UMD space} if for some (or equivalently, for 
all)
$p \in (1,\infty)$ there exists a constant $\beta>0$ such that
for every $n \geq 1$, every martingale
difference sequence $(d_j)^n_{j=1}$ in $L^p(\Omega; X)$, and every scalar-valued 
sequence
$(\varepsilon_j)^n_{j=1}$ such that $|\varepsilon_j|=1$ for each $j=1,\ldots,n$
we have
\[
\Bigl(\mathbb E \Bigl\| \sum^n_{j=1} \varepsilon_j d_j\Bigr\|^p\Bigr )^{\frac 
1p}
\leq \beta \Bigl(\mathbb E \Bigl \| \sum^n_{j=1}d_j\Bigr\|^p\Bigr )^{\frac 1p}.
\]
The least admissible constant $\beta$ is denoted by $\beta_{p,X}$ and is called 
the {\it UMD$_p$~constant} or, if the value of $p$ is understood, the {\em UMD constant}, of $X$.
It is well-known that UMD spaces obtain a large number 
of good properties, such as being reflexive. Examples of UMD 
spaces include all finite dimensional spaces and the reflexive range of 
$L^q$-spaces, Besov spaces, Sobolev spaces and Schatten class spaces. Example of 
spaces without the UMD property include all nonreflexive Banach spaces, e.g.\ 
$L^1(0,1)$, $L^{\infty}(0,1)$ and $C([0,1])$. We refer the reader to 
\cite{Burk01,HNVW1,Rubio86,Pis16} for details.

\subsection{Martingales}
Let $(\Omega,\mathcal F, \mathbb P)$ be a probability space with a filtration 
$\mathbb F = (\mathcal F_t)_{t\geq 0}$ which satisfies the usual conditions (see \cite[Definition 1.2.25]{KS} and \cite{Prot}). 
Then $\mathbb F$ 
is right-continuous and the following proposition holds:
\begin{prop}\label{prop:cadlagvers}
 Let $X$ be a Banach space. Then any martingale $M:\mathbb R_+ \times \Omega \to 
X$ admits a {\it c\`adl\`ag} version, namely there exists a version of $M$ 
which is right-continuous and has left limits.
\end{prop}

Let $t>0$. For a Banach space $X$ we define the {\it Skorohod space} $D([0, t]; 
X)$ of all right-continuous functions $f:\mathbb R_+ \to X$ with left limits. 
The following lemma follows from \cite[Problem V.6.1]{Pollard} (see also 
\cite{Skor56}).

\begin{lemma}\label{lem:cadlagcomplete}
 Let $X$ be a Banach space, $t>0$. Then $(D([0, t]; X), \|\cdot\|_{\infty})$ is 
a Banach space.
\end{lemma}

\begin{proof}[Proof of Proposition \ref{prop:cadlagvers}]
 One can find the proof in \cite[Proposition 2.2.2]{VerPhD}, but we will repeat 
it here for the convenience of the reader. Without loss of generality suppose 
that $M_{\infty}:= \lim_{t\to\infty} M_t$ exists a.s.\ and in $L^1(\Omega; X)$. Also we can assume that there 
exists $t>0$ such that $M_t = M_{\infty}$. Let $(\xi^n)_{n\geq 1}$ be a sequence 
of simple functions in $L^1(\Omega;X)$ such that $\xi^n \to M_{t}$ in 
$L^1(\Omega; X)$ as $n\to \infty$. For each $n\geq 1$ define a martingale 
$M^n:\mathbb R_+ \times \Omega \to X$ such that $M^n_s =\mathbb E(\xi^n|\mathcal 
F_s)$ for each $s\geq 0$. Fix $n\geq 1$. Since $\xi^n$ takes its values in a 
finite dimensional subspace of $X$, $M^n$ takes its values in the same finite 
dimensional subspace as well, and therefore by \cite{DM82} (or \cite[p.8]{Prot}) 
it has a c\`adl\`ag version. But $M^n_{t} = \xi^n \to M_{t}$ in $L^1(\Omega; X)$ 
as $n\to \infty$, so by the Doob maximal inequality \cite[Theorem 1.3.8(i)]{KS}, 
$M^n \to M$ in the ucp topology (the topology of the uniform convergence on 
compacts in probability). By taking an appropriate subsequence we can assume 
that $M^n \to M$ a.s.\ uniformly on $[0,t]$, and consequently, uniformly on 
$\mathbb R_+$. Therefore, by Lemma \ref{lem:cadlagcomplete} $M$ has a~c\`adl\`ag 
version. 
\end{proof}
Thanks to Proposition \ref{prop:cadlagvers} we can define $\Delta M_t$ and $M_{t-}$ for each 
$t\geq 0$,
\begin{align*}
 \Delta M_t &:= M_t - \lim_{\eps \to 0}M_{(t-\eps)\vee 0},\\
 M_{t-}&:= \lim_{\eps\to 0}M_{t-\eps},\;\;\;M_{0-}:= M_0.
\end{align*}
\subsection{Quadratic variation}
Let $(\Omega, \mathcal F, \mathbb P)$ be a probability space with a filtration 
$\mathbb F = (\mathcal F_t)_{t\geq 0}$ that
satisfies the usual conditions, $H$ be a Hilbert space. Let $M:\mathbb R_+ 
\times \Omega \to H$ be a local martingale. We define a {\it quadratic 
variation} of $M$ in the following way:
\begin{equation}\label{eq:defquadvar}
  [M]_t  := \mathbb P-\lim_{{\rm mesh}\to 0}\sum_{n=1}^N \|M(t_n)-M(t_{n-1})\|^2,
\end{equation}
where the limit in probability is taken over partitions $0= t_0 < \ldots < t_N = 
t$. The reader can find more about a quadratic variation in \cite{Kal,Prot,MP}.

\subsection{Stochastic integration}
Let $X$ be a Banach space, $H$ be a Hilbert space. For each $h\in H$, $x\in X$ 
we denote a linear operator $g\mapsto \langle g, h\rangle x$, $g\in H$, by 
$h\otimes x$. The process $\Phi: \mathbb R_+ \times \Omega \to \mathcal L(H,X)$ 
is called  \textit{elementary progressive}
with respect to the filtration $\mathbb F = (\mathcal F_t)_{t \geq 0}$ if it is 
of the form
\begin{equation}\label{eq:elprog}
 \Phi(t,\omega) = \sum_{k=1}^K\sum_{m=1}^M \mathbf 1_{(t_{k-1},t_k]\times 
B_{mk}}(t,\omega)
\sum_{n=1}^N h_n \otimes x_{kmn},\;\;\; t\geq 0, \omega \in \Omega,
\end{equation}
where $0 \leq t_0 < \ldots < t_K <\infty$, for each $k = 1,\ldots, K$ the sets
$B_{1k},\ldots,B_{Mk}$ are in $F_{t_{k-1}}$ and vectors $h_1,\ldots,h_N$ are 
orthogonal.

\smallskip

Let $M:\mathbb R_+ \times \Omega \to H$ be a martingale. Then we define a {\it 
stochastic integral} $\Phi \cdot M:\mathbb R_+ \times \Omega \to X$ of $\Phi$ 
with respect to $M$ in the following way:
\begin{equation}\label{eq:defofstochintwrtM}
 (\Phi \cdot M)_t = \sum_{k=1}^K\sum_{m=1}^M \mathbf 1_{B_{mk}}
\sum_{n=1}^N \langle(M(t_k\wedge t)- M(t_{k-1}\wedge t)), h_n\rangle 
x_{kmn},\;\; t\geq 0.
\end{equation}

The reader can find more on stochastic integration in a finite dimensional case in \cite{Kal}. The following lemma is a multidimensional version of \cite[Theorem 
26.6(v)]{Kal}.
\begin{lemma}\label{lemma:KunWat}
 Let $d$ be a natural number, $H$ be a $d$-dimensional Hilbert space, $M:\mathbb R_+ \times \Omega \to H$ be a 
martingale, $\Phi:\mathbb R_+ \times \Omega \to \mathcal L(H, \mathbb R)$ be 
elementary progressive. Then $[\Phi \cdot M] \lesssim_{d}\|\Phi\|^2 
\cdot [M]$ a.s.
\end{lemma}

\begin{proof}
 Let $(h_n)_{n=1}^d$ be an orthogonal basis of $H$, $\Phi_1,\ldots,\Phi_d:\mathbb R_+ \times \Omega \to \mathbb R$ be such that $\Phi = \sum_{n=1}^d \Phi_n h_n$, and $M_1, \ldots,M_d:\mathbb R_+\times \Omega \to \mathbb R$ be martingales such that $M = \sum_{n=1}^d M_n h_n$. Notice that thanks to the definition of a quadratic variation \eqref{eq:defquadvar} one has that $[M] = [M_1]+\cdots +[M_d]$. Then since a quadratic variation is a positive-definite quadratic form (see \cite[Theorem 26.6]{Kal}), thanks to \cite[Theorem~26.6(v)]{Kal} one has for each $t\geq 0$ a.s.,
 \begin{align*}
  [\Phi \cdot M]_t =  [\Phi_1 \cdot M_1 + \cdots + \Phi_d\cdot M_d]_t &\lesssim_d [\Phi_1 \cdot M_1]_t + \cdots + [\Phi_d\cdot M_d]_t\\
  &=(\|\Phi_1\|^2 \cdot [M_1])_t + \cdots +(\|\Phi_d\|^2 \cdot [M_d])_t\\
  &\lesssim_d (\|\Phi\|^2 \cdot [M])_t.
 \end{align*}
\end{proof}

Using Lemma \ref{lemma:KunWat} one can extend stochastic integral to the case of general $\Phi$. In particular, the following lemma on stochastic integration can be shown.

\begin{lemma}\label{lemma:stochintmoment}
 Let $d$ be a natural number, $H$ be a $d$-dimensional Hilbert space, $p\in (1,\infty)$, $M, N:\mathbb R_+ \times \Omega 
\to H$ be $L^p$-martingales, $F:H \to H$ be a measurable function such that 
$\|F(h)\|\leq C \|h\|^{p-1}$ for each $h\in H$ and some $C>0$. Let $N_-:\mathbb 
R_+ \times \Omega \to H$ be such that $(N_-)_t = N_{t-}$ for each $t\geq 0$. 
Then $F(N_{-})\cdot M$ is a martingale and for each $t\geq 0$,
 \begin{equation}\label{eq:stochintmoment}
  \mathbb E |(F(N_-)\cdot M)_t|\lesssim_{p, d} C(\mathbb E \|N_t\|^p)^{\frac 
{p-1}p} (\mathbb E \|M_t\|^p)^{\frac {1}p}.
 \end{equation}
\end{lemma}

\begin{proof}
 First notice that $F(N_-)$ is predictable. Therefore, thanks to Lemma 
\ref{lemma:KunWat} and \cite[Theorem 26.12]{Kal},  in order to 
prove that $F(N_{-})$ is stochastically integrable with respect to $M$ and that $F(N_{-})\cdot M$ 
is a martingale it is sufficient to show 
that $\mathbb E (\|F(N_-)\|^2\cdot [M])_t^{\frac 12}<\infty$. Without loss of 
generality suppose that $M_0 = N_0=0$ a.s.\ and $C=1$. Then
 \begin{align}\label{eq:stochintmomentproof}
 \mathbb E (\|F(N_-)\|^2\cdot [M])_t^{\frac 12} &\leq \mathbb E 
(\|N_{t-}\|^{2(p-1)}\cdot [M]_t)^{\frac 12}
  \leq\mathbb E\Bigl( \sup_{0\leq s\leq t}\|N_s\|^{p-1} [M]_t^{\frac 
12}\Bigr)\nonumber\\
  &\stackrel{(i)}\leq (\mathbb E \sup_{0\leq s\leq t}\|N_s\|^{p})^{\frac {p-1}p} 
(\mathbb E [M]_t^{\frac p2})^{\frac {1}p} \\
  &\stackrel{(ii)}\lesssim_p(\mathbb E \|N_t\|^p)^{\frac {p-1}p} (\mathbb E 
\|M_t\|^p)^{\frac {1}p}<\infty,\nonumber
 \end{align}
where $(i)$~follows from the H\"older inequality, and $(ii)$ holds thanks to 
\cite[Theorem 26.12]{Kal} and \cite[Theorem 1.3.8(iv)]{KS}.

Now let us show \eqref{eq:stochintmoment}:
\begin{align*}
  \mathbb E |(F(N_-)\cdot M)_t|&\stackrel{(i)}\lesssim_p \mathbb E 
[F(N_-)\cdot M]^{\frac 12}_t \stackrel{(ii)}\lesssim_d \mathbb E 
(\|F(N_-)\|^2\cdot [M])_t^{\frac 12}\\
  &\stackrel{(iii)}\lesssim_p(\mathbb E \|N_t\|^p)^{\frac {p-1}p} 
(\mathbb E \|M_t\|^p)^{\frac {1}p}.
\end{align*}
Here $(i)$ follows from \cite[Theorem 26.12]{Kal}, $(ii)$ holds 
thanks to Lemma \ref{lemma:KunWat}, and $(iii)$ follows 
from \eqref{eq:stochintmomentproof}.
\end{proof}

\section{UMD Banach spaces and weak differential subordination}

 From now on the
scalar field $\mathbb K$ can be either $\mathbb R$ or 
$\mathbb C$.

\subsection{Discrete case}
In this section we assume that $X$ is a Banach space over the scalar field 
$\mathbb K$ and with a separable dual $X^*$.
 Let $(\Omega, \mathcal F, \mathbb P)$ be a complete probability space with 
filtration
$\mathbb F := (\mathcal F_n)_{n \geq 0}$, $\mathcal F_{0} = \{\varnothing, 
\Omega\}$.

\begin{defi}
 Let $(f_n)_{n\geq 0}$, $(g_n)_{n\geq 0}$ be $X$-valued local martingales. For 
each $n\geq 1$ we define $df_n := f_n - f_{n-1}$, $dg_n := g_n - g_{n-1}$.
 \begin{itemize}
  \item [(i)]$g$ is {\em differentially subordinated} to $f$ if one has that 
$\|dg_n\|\leq \|df_n\|$ a.s.\ for all $n\geq 1$ and $\|g_0\|\leq \|f_0\|$ a.s.
  \item[(ii)]$g$ is {\em weakly differentially subordinated} to $f$ if for each 
$x^* \in X^*$ one has that $|\langle dg_n,x^*\rangle|\leq|\langle 
df_n,x^*\rangle|$ a.s.\ for all $n\geq 1$ and $|\langle g_0,x^*\rangle|\leq 
|\langle f_0,x^*\rangle|$ a.s.
 \end{itemize}
\end{defi}

 The following characterization of Hilbert spaces can be found in 
\cite[Theorem~3.24(i)]{Os12}:

\begin{theorem}
 A Banach space $X$ is isomorphic to a Hilbert space if and 
only if for some (equivalently, for all) $1<p<\infty$ there exists a constant 
$\alpha_{p,X}>0$ such that for any pair of $X$-valued local martingales 
$(f_n)_{n\geq 0}$, $(g_n)_{n\geq 0}$ such that $g$ is differentially 
subordinated to $f$ one has that \begin{equation}\label{eq:forpnormsubprop}
        \mathbb E \|g_n\|^p \leq \alpha_{p,X}^p \mathbb E \|f_n\|^p     
\end{equation}
 for each $n\geq 1$.
\end{theorem}

By the Pettis measurability theorem 
\cite[Theorem 1.1.20]{HNVW1}, we may assume that $X$ is separable. Then weak 
differential subordination implies differential subordination. 
Indeed, let $(x_k)_{k\geq 1}$ be a dense subset of $X$, $(x_k^*)_{k\geq 1}$ be 
 a sequence of linear functionals on $X$ such that $\langle x_k, x_k^*\rangle = 
\|x_k\|$ and $\|x_k^*\|=1$ for each $k\geq 1$ (such a sequence exists by the 
Hahn-Banach theorem). Let $(g_n)_{n\geq 0}$ be weakly differentially 
subordinated to $(f_n)_{n\geq 0}$. Then for each $n\geq 1$~a.s.
\[
 \|dg_n\| = \sup_{k\geq 1}|\langle dg_n, x^*_k\rangle| \leq  \sup_{k\geq 
1}|\langle df_n, x^*_k\rangle|= \|df_n\|.
\]
By the same reasoning $\|g_0\|\leq \|f_0\|$ a.s.
This means that the weak differential subordination property is more restrictive 
than the differential subordination property. Therefore, under the weak differential subordination, one could expect that 
the assertions of the type \eqref{eq:forpnormsubprop} characterize a broader class of 
Banach spaces $X$.
Actually we will prove the following theorem, which extends 
\cite[Theorem~2]{Burk85a} to the UMD case.

\begin{theorem}\label{thm:forseq}
A Banach space $X$ is a UMD space if and only if for some (equi\-va\-lent\-ly, for all) $1<p<\infty$ there exists a constant 
$\beta>0$ such that for all \black $X$-valued local martingales $(f_n)_{n\geq 0}$ and 
$(g_n)_{n\geq 0}$ such that $g$ is weakly differentially subordinated to $f$ one 
has
 \begin{equation}\label{eq:forpnorm}
  \mathbb E \|g_n\|^p\leq \beta^p \mathbb E\|f_n\|^p,\;\;\; n\geq 1.
 \end{equation}
If this is the case then the smallest admissible $\beta$ is the UMD constant 
$\beta_{p,X}$.
\end{theorem}

Theorem \ref{thm:introweakdiffsub} is contained in this result as a special 
case. 

The proof of Theorem \ref{thm:forseq} consists of several steps.

\begin{prop}\label{thm:discweakdiffsub}
 Let $X$ be a Banach space. Let $(f_n)_{n\geq 0}$, $(g_n)_{n\geq 0}$ be two 
$X$-valued local martingales. Then $g$ is weakly differentially subordinated to 
$f$ if and only if there exists an adapted scalar-valued process $(a_n)_{n\geq 
0}$ such that $|a_n|\leq 1$ a.s.\ for all $n\geq 1$, $dg_n = a_ndf_n$ a.s.\ and 
$g_0 = a_0f_0$ a.s.
\end{prop}

For the proof we will need two lemmas.

\begin{lemma}\label{lemma:linfuncwiththesameker}
 Let $X$ be a Banach space, $\ell_1,\ell_2 \in X^*$ be such that 
$\text{ker}(\ell_1) \subset \text{ker}(\ell_2)$. Then there exists $a\in \mathbb 
K$ such that $\ell_2 = a\ell_1$.
\end{lemma}

\begin{proof}
 If $\ell_2=0$, then the assertion is obvious and one can take $a=0$. Suppose 
that $\ell_2 \neq0$. Then $\text{codim}(\text{ker}(\ell_2)) = 1$ (see 
\cite[p.80]{KF1}), and there exists $x_0 \in X\setminus \text{ker}(\ell_2)$ such 
that $x_0 \oplus \text{ker}(\ell_2) = X$. Notice that since 
$\text{codim}(\text{ker}(\ell_1)) \leq 1$ and $\text{ker}(\ell_1) \subset 
\text{ker}(\ell_2)$, one can easily conclude that $\text{ker}(\ell_1) = 
\text{ker}(\ell_2)$. Let $a = \ell_2(x_0)/\ell_1(x_0)$. Fix $y \in X$. Then 
there exists $\lambda \in \mathbb K$ such that $y-\lambda x_0 \in 
\text{ker}(\ell_1) =\text{ker}(\ell_2)$. Therefore
 \[
  \ell_2(y) = \ell_2(\lambda x_0) + \ell_2(y-\lambda x_0) = a\ell_1(\lambda 
x_0)+ a\ell_1(y-\lambda x_0) = a\ell_1(y),
 \]
 hence $\ell_2 = a\ell_1$.
\end{proof}

\begin{lemma}\label{lemma:ajustformeasfunc}
 Let $X$ be a Banach space, $(S,\Sigma, \mu)$ be a measure space. Let $f,g:S \to 
X$ be strongly measurable such that $|\langle g, x^*\rangle|\leq |\langle f, 
x^*\rangle|$ $\mu$-a.s.\ for each $x^* \in X^*$. Then there exists a measurable 
function $a:S\to\mathbb K$ such that $\|a\|_{\infty}\leq 1$ and $g=af$.
\end{lemma}

\begin{proof}
By the Pettis measurability theorem \cite[Theorem 1.1.20]{HNVW1} we can assume 
$X$ to be separable. Let $(x_m)_{m\geq 1}$ be a dense subset of $X$. By the 
Hahn-Banach theorem we can find a sequence $(x_m^*)_{m\geq 1}$ of linear 
functionals on $X$ such that $\langle x_m, x_m^*\rangle = \|x_m\|$ and 
$\|x_m^*\|=1$ for each $m\geq 1$. Let $Y_0 = \mathbb Q-\text{span}(x^*_1,x^*_2, 
\ldots)$, and let $Y = \overline{\text{span}(x^*_1,x^*_2, \ldots)}$ be a 
separable closed subspace of $X^*$. Then $X\hookrightarrow Y^*$ isometrically. 
Fix a set of full measure $S_0$ such that for all $x^* \in Y_0$, $|\langle 
g,x^*\rangle|\leq |\langle f,x^*\rangle|$ on $S_0$. Fix $x^* \in Y$. Let 
$(y_{k})_{k\geq 1}$ be a sequence in $Y_0$ such that $y_{k} \to x^*$ in $Y$ as 
$k\to \infty$. Then on $S_0$ we have that $|\langle g,y^*_{k}\rangle| \to 
|\langle g,x^*\rangle|$ and $|\langle f,y_{k}\rangle| \to |\langle 
f,x^*\rangle|$. Consequently for each $s\in S_0$,
 \begin{equation}\label{eq:discweakdiffsub}
  |\langle g(s),x^*\rangle|\leq |\langle f(s),x^*\rangle|, \;\;\;x^* \in Y.
 \end{equation}
 Therefore the linear functionals $f(s), g(s) \in X \hookrightarrow Y^*$ are 
such that $\ker g(s)\subset \ker f(s)$, and hence by Lemma 
\ref{lemma:linfuncwiththesameker} there exist $a(s)$ defined for each fixed $s 
\in S_0$ such that $g(s) = a(s)f(s)$. By \eqref{eq:discweakdiffsub} one has that 
$|a(s)|\leq 1$. 
 
Let us construct a measurable version of $a$. $Y_0$ is countable since it is a 
$\mathbb Q-\text{span}$ of a countable set. Let $Y_0 = (y_m)_{m\geq 1}$. For 
each $m> 1$ construct $A_m \in \Sigma$ as follows:
 \[
  A_m = \{s \in S: \langle g(s), y_{m} \rangle \neq 0,\langle g(s), y_{m-1} 
\rangle = 0,\ldots,\langle g(s), y_1 \rangle  = 0\}
 \]
 and put $A_1 = \{s\in S: \langle g(s), y_{1} \rangle \neq 0\}$. Obviously on 
the set $S\setminus \cup_{m=1}^{\infty} A_m$ one has that $g = 0$, so one can 
redefine $a := 0$ on $S\setminus \cup_{m=1}^{\infty} A_m$. For each $m\geq 1$ we 
redefine $a:=\frac{\langle g, y_m\rangle}{\langle f, y_m\rangle}$ on $A_m$. Then 
$a$ constructed in such a way is $\Sigma$-measurable.
\end{proof}

\begin{proof}[Proof of Proposition \ref{thm:discweakdiffsub}]
The proposition follows from Lemma \ref{lemma:ajustformeasfunc}: the assumption 
of this lemma holds for $df_n$ and $dg_n$ for any $n\geq1$, and for $f_0$ and 
$g_0$. So according to Lemma \ref{lemma:ajustformeasfunc} there exists a 
sequence $(a_n)_{n\geq 0}$ which is a.s.\ bounded by $1$, such that $dg_n = a_n 
df_n$ for each $n\geq 1$ and $g_0 = a_0f_0$ a.s.\ Moreover, again thanks to 
Lemma~\ref{lemma:ajustformeasfunc}, $a_n$ is $\mathcal F_n$-measurable, so 
$(a_n)_{n\geq 0}$ is adapted.
\end{proof}

\begin{defi}
 Let $E$ be a linear space over the scalar field $\mathbb K$.
 \begin{itemize}
  \item [(i)]A function $f:E \to \mathbb R$ is called  {\em convex} if for each $x, y 
\in E$, $\lambda \in [0,1]$ one has that $f(\lambda x+(1-\lambda)y)\leq \lambda 
f(x)+(1-\lambda)f(y)$.
  \item[(ii)]A function $f:E \to \mathbb R$ is called {\em concave} if for each $x, y 
\in E$, $\lambda \in [0,1]$ one has that $f(\lambda x+(1-\lambda)y)\geq \lambda 
f(x)+(1-\lambda)f(y)$.
  \item[(iii)] A function $f:E\times E \to \mathbb R$ is called {\em biconcave} 
if for each $x, y \in E$ one has that the mappings $e\mapsto f(x, e)$ and 
$e\mapsto f(e,y)$ are concave.
  \item[(iv)] A function $f:E\times E \to \mathbb R$ is called {\em zigzag-concave} if 
for each $x, y\in E$ and $\eps \in \mathbb K$, $|\eps| \leq 1$ the function 
$z\mapsto f(x+z, y+\eps z)$ is concave.
 \end{itemize}
\end{defi}

Note that our definition of zigzag-concavity is a bit different from the classical 
one (e.g.\ as in \cite{HNVW1}): usually one sets in the definition $|\eps|=1$. The reader should pay 
attention to this extension: thanks to this additional property Theorem \ref{thm:Burkholder} below is more general than \cite[Theorem 
4.5.6]{HNVW1}. This improvement will later allow us to prove the main theorem of 
this section.

In \cite{Burk86} Burkholder showed that the UMD property is equivalent to the 
existence of a certain biconcave function $V:X\times X \to \mathbb R$. With a 
slight variation of his argument (see Remark \ref{rem:fromUtoV}) one can also 
show the equivalence with the existence of a certain zigzag-concave function with a 
better structure.

\begin{theorem}[Burkholder]\label{thm:Burkholder}
 For a Banach space $X$ the following are equivalent
 \begin{enumerate}
  \item $X$ is a UMD Banach space;
  \item for each $p\in (1,\infty)$ there exists a constant $\beta>0$ and a 
zigzag-concave function $U:X\times X \to \mathbb R$ such that
  \begin{equation}\label{eq:ineqonU}
     U(x,y)\geq \|y\|^p - \beta^p\|x\|^p,\;\;\;x,y\in X.
  \end{equation}
 \end{enumerate}
 The smallest admissible $\beta$ for which such $U$ exists is $\beta_{p, X}$.
\end{theorem}
\begin{proof}
 The proof is essentially the same as the one given in \cite[Theorem 
4.5.6]{HNVW1}, but the construction of $U$ is a bit different. The only 
difference is allowing $|\eps|\leq 1$ instead of $|\eps|=1$ for the appropriate 
scalars $\eps$.

 For each $x, y\in X$ we define $\mathbb S(x,y)$ as a set of all pairs $(f,g)$ 
of discrete martingales such that
 \begin{enumerate}
  \item $f_0\equiv x$, $g_0\equiv y$;
  \item there exists $N\geq 0$ such that $df_n\equiv 0$, $dg_n\equiv 0$ for 
$n\geq N$;
  \item $(dg_n)_{n\geq 1} = (\eps_ndf_n)_{n\geq 1}$ for some sequence of scalars 
$(\eps_n)_{n\geq 1}$ such that $|\eps_n|\mathbf{\leq} 1$ for each $n\geq 1$.
 \end{enumerate}
Then we define $U:X\times X \to \mathbb R \cup \{\infty\}$ as follows:
\begin{equation}\label{eq:formulaofU}
 U(x,y) := \sup\bigl\{\mathbb E(\|g_{\infty}\|^p - \beta^p \|f_{\infty}\|^p): 
(f,g)\in \mathbb S(x,y)\bigr\}.
\end{equation}
The rest of the proof repeats the one given in \cite[Theorem 4.5.6]{HNVW1}.
\end{proof}

\begin{remark}
 Notice that function $U$ constructed above coincides with the one in the proof of 
\cite[Theorem 4.5.6]{HNVW1}. This is due to the fact that the function 
 $$
 (\eps_n)_{n=1}^N \mapsto \Bigl(\mathbb E \Bigl\|g_0 + \sum_{n=1}^N\eps_n 
df_n\Bigr\|^p\Bigr)^{\frac 1p}
 $$
 is convex on the $\mathbb K$-cube $\{(\eps_n)_{n=1}^N: 
|\eps_1|,\ldots,|\eps_N|\leq 1\}$ because of the triangle inequality, therefore 
it takes its supremum on the set of the domain endpoints, namely on the set 
$\{(\eps_n)_{n=1}^N: |\eps_1|,\ldots,|\eps_N|= 1\}$.
\end{remark}

\begin{remark}\label{rem:propsofU}
 Analogously to \cite[(4.31)]{HNVW1} by \eqref{eq:formulaofU} we have 
that $U(\alpha x, \alpha y) = |\alpha|^p U(x, y)$ for each $x, y\in X$, $\alpha 
\in \mathbb K$. Therefore $U(0, 0) = 0$, and hence for each $x\in X$ and each 
scalar $\eps$ such that $|\eps|\leq 1$, by the zigzag-concavity of $U$ in the 
point~$(0,0)$
 \begin{equation}\label{eq:propofU(x,epsx)}
    U(x, \eps x) = \frac 12U(0+x,0+ \eps x) + \frac 12U(0-x, 0-\eps x)\leq 
U(0,0)=0.
 \end{equation}
Let $\xi, \eta\in L^0(\Omega; X)$ be such that $|\langle \eta, x^*\rangle|\leq 
|\langle \xi, x^*\rangle|$ for each $x^*\in X^*$ a.s. Then thanks to Lemma 
\ref{lemma:ajustformeasfunc} and \eqref{eq:propofU(x,epsx)}, $U(\xi,\eta)\leq 0$ 
a.s.
\end{remark}

\begin{remark}\label{rem:fromUtoV}
 For each zigzag-concave function $U:X\times X \to \mathbb R$ one can construct 
a biconcave function $V:X\times X \to \mathbb R$ as follows:
\begin{equation}\label{eq:defofV}
 V(x,y) = U\Bigl(\frac{x-y}{2},\frac{x+y}2\Bigr),\;\;\; x,y\in X.
\end{equation}
Indeed, by the definition of $U$, for each $x,y\in X$ the functions
\begin{align*}
 z\mapsto V(x+z,y) &= U\Bigl(\frac{x-y}{2} + \frac{z}2,\frac{x+y}2 + 
\frac{z}2\Bigr),\\
  z\mapsto V(x,y+z) &= U\Bigl(\frac{x-y}{2} - \frac{z}2,\frac{x+y}2 + 
\frac{z}2\Bigr)
\end{align*}
are concave. Moreover, for each $x,y\in X$ and $a,b \in \mathbb K$ such that 
$|a+b|\leq |a-b|$ one has that the function
$$
z\mapsto V(x+az,y+bz) = U\Bigl(\frac{x-y}{2} + \frac{(a-b)z}2,\frac{x+y}2 + 
\frac{(a+b)z}2\Bigr)
$$
is concave since $\bigl| \frac {a+b}{a-b}\bigr|\leq 1$.
\end{remark}

\begin{remark}\label{rem:contofUandV}
 Due to the explicit representation \eqref{eq:formulaofU} of $U$ we can show 
that for each $x_1, x_2, y_1, y_2\in X$,
 $$
 |U(x_1, y_1) - U(x_2, y_2)|\leq \|x_1-x_2\|^p + \beta_{p, X}^p \|y_1-y_2\|^p.
 $$
 Therefore $U$ is continuous, and consequently $V$ is continuous as well.
\end{remark}

\begin{remark}\label{rem:findimXforV}
Notice that if $X$ is finite dimensional then by Theorem 2.20 and 
Proposition 2.21 in \cite{FolHarm} there exists a unique  translation-invariant 
measure $\lambda_X$ on $X$ such that $\lambda_X(\mathbb B_X) = 1$ for the unit 
ball $\mathbb B_X$ of $X$. We will call $\lambda_X$ a {\it Lebesgue measure}. 
Thanks to the Alexandrov theorem \cite[Theorem 6.4.1]{EG} $x\mapsto V(x, y)$ and 
$y\mapsto V(x,y)$ are a.s.\ Fr\'echet differentiable with respect to 
$\lambda_X$, and by \cite[Proposition 3.1]{JT} and Remark 
\ref{rem:contofUandV} for a.a.\ $(x,y)\in X\times X$ for each $u,v\in X$ there 
exists the directional derivative $\frac{\partial V(x+tu,y+tv)}{\partial t}$. 
Moreover, 
\begin{equation}\label{eq:dirderivative}
 \frac{\partial V(x+tu,y+tv)}{\partial t} = \langle \partial_x V(x,y), 
u\rangle+\langle \partial_y V(x,y), v\rangle,
\end{equation}
where $\partial_x V$ and $\partial_y V$ are the corresponding Fr\'echet 
derivatives with respect to the first and the second variable. Thanks to 
\eqref{eq:dirderivative} and Remark \ref{rem:fromUtoV} one obtains that for 
a.e.\ $(x,y)\in X\times X$, for all $z\in X$ and $a,b\in \mathbb K$ such that 
$|a+b|\leq |a-b|$,
\begin{equation}\label{eq:dirderivativeV}
\begin{split}
 V(x+az,y+bz)&\leq V(x,y) + \frac{\partial V(x+atz,y+btz)}{\partial t}\\
 &= V(x,y)+a\langle \partial_x V(x,y), z\rangle+b\langle \partial_y V(x,y), 
z\rangle.
 \end{split}
\end{equation}
\end{remark}

\begin{lemma}\label{lem:V_xV_yarenice}
 Let $X$ be a finite dimensional Banach space, $V:X\times X\to \mathbb R$ be as 
defined in \eqref{eq:defofV}. Then there exists $C>0$ which depends only on $V$ 
such that for a.e.\ pair $x, y\in X$,
 $$
 \|\partial_xV(x, y)\|, \|\partial_yV(x, y)\| \leq C(\|x\|^{p-1} + \|y\|^{p-1}).
 $$
\end{lemma}

\begin{proof}
 We show the inequality only for $\partial_xV$, the proof for $\partial_yV$ being 
analogous. First we prove that there exists $C>0$ such that $\|\partial_xV(x, 
y)\|\leq C$ for a.e.\ $x, y\in X$ such that $\|x\|, \|y\|\leq 1$. Let us show this 
by contradiction. Suppose that such $C$ does not exist. Since $V$ is continuous 
by Remark \ref{rem:contofUandV}, and since a unit ball in $X$ is a compact set, 
there exists $K>0$ such that $|V(x, y)|<K$ for all $x, y\in X$ such that $\|x\|, 
\|y\|\leq 2$. Let $x_0, y_0\in X$ be such that $\|x_0\|, \|y_0\| \leq 1$ and 
$\|\partial_xV(x_0, y_0)\|>3K$. Therefore there exists $z\in X$ such that $\|z\| 
= 1$ and $\langle \partial_xV(x_0, y_0), z\rangle<-3K$. Hence we have that 
$\|x_0 + z\|\leq2$ and because of the concavity of $V$ in the first variable
 \[
  V(x_0 + z, y_0) \leq V(x_0, y_0) + \langle \partial_xV(x_0, y_0), z\rangle\leq 
K - 3K \leq -2K.
 \]
Consequently, $|V(x_0 + z, y_0)|>K$, which contradicts with our suggestion.

Now fix $C>0$ such that $|\partial_xV(x, y)|\leq C$ for all $x, y\in X$ such 
that $\|x\|, \|y\|\leq 1$. Fix $x, y\in X$. Without loss of generality assume that 
$\|x\|\geq \|y\|$. Let $L=\|x\|$. Then $\|\partial_xV\bigl(\frac xL, \frac 
yL\bigr)\| \leq C$. Let $z\in X$ be such that $\|z\|=1$. Then by Remark~\ref{rem:propsofU},
\begin{align*}
  |\langle \partial_xV (x, y), z\rangle|\!&=\!\biggl| \lim_{t\to 0}\frac 
{V(x\!+\!tz, y)\!-\!V(x, y)}{t}\biggr
  |\!=\!\biggl| \lim_{t\to 0}\frac {L^pV(\frac xL\!+\!\frac tL z, \frac 
yL)\!-\!L^pV(\frac xL, \frac yL)}{L\frac tL}\biggr|\\
  &=L^{p-1}\biggl|\lim_{t\to 0}\frac {V(\frac xL\!+\!tz, \frac yL)\!-\!V(\frac 
xL, \frac yL)}{t}\biggr|\!=\!L^{p-1}\Bigl|\Bigl\langle \partial_xV \Bigl(\frac 
xL, \frac yL\Bigr), z\Bigr\rangle\Bigr|\\
  &\leq L^{p-1} C \leq C(\|x\|^{p-1} + \|y\|^{p-1}).
\end{align*}
Therefore since $z$ was arbitrary, $\| \partial_xV (x, y)\|\leq C(\|x\|^{p-1} + 
\|y\|^{p-1})$. The case $\|x\|<\|y\|$ can be done in the same way.
\end{proof}

\bigskip

\begin{lemma}\label{lemma:abscontmart}
 Let $X$ be a finite dimensional Banach space, $1<p<\infty$, $(f_n)_{n\geq 0}$, 
$(g_n)_{n\geq 0}$ be $X$-valued martingales on a probability space $(\Omega, 
\mathcal F, \mathbb P)$ with a filtration $\mathbb F = (\mathcal F_n)_{n\geq 0}$ 
and assume that $(g_n)_{n\geq 0}$ is weakly differentially subordinated to 
$(f_n)_{n\geq 0}$. Let $Y=X\oplus \mathbb R$ be the Banach space with the norm as follows: 
$$
\|(x,r)\|_Y:= (\|x\|_X^p + |r|^p)^{\frac 1p},\;\;\; x\in X, r\in \mathbb R.
$$
Then there exist two sequences $(f^m)_{m\geq 1}$ and $(g^m)_{m\geq 1}$ of 
$Y$-valued martingales on an enlarged probability space $(\overline{\Omega}, 
\overline{\mathcal F}, \overline{\mathbb P})$ with an enlarged filtration 
$\overline{\mathbb F} = (\overline{\mathcal F}_n)_{n\geq 0}$ such that
 \begin{enumerate}
  \item $f^m_n$, $g^m_n$ have absolutely continuous distributions with respect 
to the Lebesgue measure on $Y$ for each $m\geq 1$ and $n\geq 0$;
  \item $f^m_n \to (f_n,0)$, $g^m_n \to (g_n,0)$ pointwise as $m\to \infty$ for 
each $n\geq 0$;
  \item if for some $n\geq 0$ $\mathbb E\|f_n\|^p<\infty$, then for each $m\geq 
1$ one has that $\mathbb E\|f^m_n\|^p<\infty$ and $\mathbb 
E\|f^m_n-(f_n,0)\|^p\to 0$ as $m\to \infty$;
  \item if for some $n\geq 0$ $\mathbb E\|g_n\|^p<\infty$, then for each $m\geq 
1$ one has that $\mathbb E\|g^m_n\|^p<\infty$ and $\mathbb 
E\|g^m_n-(g_n,0)\|^p\to 0$ as $m\to \infty$;
  \item for each $m\geq 1$ we have that $(g^m_n)_{n\geq 0}$ is weakly 
differentially subordinated to $(f^m_n)_{n\geq 0}$.
 \end{enumerate}
\end{lemma}

\begin{proof}
First of all let us show that we may assume that $f_0$ and $g_0$ are nonzero 
a.s.\ For this purpose we can modify $f_0$ and $g_0$ as follows:
\begin{align*}
 f_0^{\eps} &= f_0 + \eps x \mathbf 1_{f_0 = 0},\\
 g_0^{\eps} &= g_0 + \eps x \mathbf 1_{f_0 = 0} + \eps f_0 \mathbf 1_{g_0 = 0, 
f_0 \neq 0},
\end{align*}
where $\eps>0$ is arbitrary and $x\in X$ is fixed. This small perturbation does 
not destroy the weak differential subordination property. Moreover, if we let 
$f_n^{\eps}:= f_0^{\eps} + \sum_{k=1}^n df_k$, $g_n^{\eps}:= g_0^{\eps} + 
\sum_{k=1}^n dg_k$ for any $n\geq 1$, then $f_n^{\eps}\to f_n$ and 
$g_n^{\eps}\to g_n$ a.s., and $f_n^{\eps}- f_n \to 0$ and $g_n^{\eps}- g_n\to 0$ 
in $L^p(\Omega; X)$ as $\eps \to 0$.

From now we assume that $f_0$ and $g_0$ are nonzero a.s.\ This in fact means that random variable $a_0$ 
from Proposition \ref{thm:discweakdiffsub} is nonzero a.s.\ as well.
 Let  $\mathbb B_Y$ be the unit ball of $Y$, $(\mathbb B_Y, \mathcal B(\mathbb 
B_Y), \hat{\mathbb P})$ be a probability space such that $\hat{\mathbb P}:= 
\lambda_Y|_{\mathbb B_Y}$ has the uniform Lebesgue distribution on $\mathbb B_Y$ 
(see Remark \ref{rem:findimXforV}). Fix some scalar product $\langle 
\cdot,\cdot\rangle:Y\times Y\to \mathbb R$ in $Y$. We will construct a random 
operator $T:\mathbb B_Y\to\mathcal L(Y)$ as follows: 
 $$
 T(b,y) := \langle b,y\rangle b\;\;\;\;\;\;\; b\in \mathbb B_Y, y\in Y.
 $$
 Note that for each fixed $b\in \mathbb B_Y$ the mapping $y\mapsto \langle 
b,y\rangle b$ is a linear operator on $Y$. Moreover, 
 \begin{equation}\label{eq:abscontmartnormest}
  \sup_{b\in \mathbb B_Y}\|T(b,\cdot)\|_{\mathcal L(Y)}< \infty.
 \end{equation}
 Now let $(\overline {\Omega}, \overline{\mathcal F}, \overline {\mathbb P}) := 
(\Omega \times \mathbb B_Y, \mathcal F \otimes \mathcal B(\mathbb B_Y), \mathbb 
P \otimes \hat{\mathbb P})$. For each $m\geq 1$ define an operator-valued 
function $Q_m:\overline {\Omega}\to \mathcal L(Y)$ as follows: $Q_m := I + \frac 
1m T$. 
 
 Fix $\eps>0$. For each $n\geq 0$ define $\tilde f^{\eps}_n := (f_n,\eps)$, 
$\tilde g^{\eps}_n := (g_n,\eps a_0)$. Then $(\tilde f^{\eps}_n)_{n\geq 0}$ and 
$(\tilde g^{\eps}_n)_{n\geq 0}$ are $Y$-valued martingales which are nonzero 
a.s.\ for each $n\geq 0$ and are such that $(\tilde g^{\eps}_n)_{n\geq 0}$ is 
weakly differentially subordinated to $(\tilde f^{\eps}_n)_{n\geq 0}$. 
 For each $m\geq 1$ define $Y$-valued martingales $f^m$ and $g^m$ in the 
following way:
 \begin{align*}
  f^m_n &:= Q_m \tilde f^{\eps}_n,\;\;m\geq  1, n\geq 0,\\
   g^m_n &:= Q_m \tilde g^{\eps}_n,\;\;m\geq  1, n\geq 0.
 \end{align*}
 Let us illustrate that for each $m\geq 1$, $f^m$ and $g^m$ are martingales with 
respect to the filtration $\overline{\mathbb F} = (\overline{\mathcal 
F}_n)_{n\geq 0}:=(\mathcal F_n\otimes \mathcal B(\mathbb B_Y))_{t\geq 0}$: for 
each $n\geq 1$ we have
 \begin{align*}
   \mathbb E(f^m_n| \overline{\mathcal F}_{n-1}) = \mathbb E(Q_m\tilde 
f^{\eps}_n| \mathcal F_{n-1}\otimes \mathcal B(\mathbb B_Y))
 &\stackrel{(i)}=Q_m\mathbb E(\tilde f^{\eps}_n| \mathcal F_{n-1}\otimes 
\mathcal B(\mathbb B_Y))\\
 &\stackrel{(ii)}= Q_m\tilde f^{\eps}_{n-1} =f^m_{n-1},
 \end{align*}
 where $(i)$ holds since $Q_m$ is $\mathcal B(\mathbb B_Y)$-measurable, and 
$(ii)$ holds since $\tilde f^{\eps}_n$ is independent of $\mathcal B(\mathbb 
B_Y)$.
 The same can be proven for $g^m$.
Thanks to \eqref{eq:abscontmartnormest} one has that $\lim_{m\to 
\infty}\sup_{b\in \mathbb B_Y}\|Q_m - I\|_{\mathcal L (Y)}=0$ and hence (2), (3) 
and (4) hold for $\tilde f^{\eps}$ and $\tilde g^{\eps}$. 

Let us prove (5). For each $m\geq 1$ and $n\geq 1$ one has:
\[
 dg^m_n = dQ_m\tilde g^{\eps}_n = dQ_ma_n\tilde f^{\eps}_n = a_ndQ_m\tilde 
f^{\eps}_n =  a_ndf^m_n.
\]
The same holds for $g^m_0$ and $f^m_0$.
 
 Now we will show (1). Let us fix a set $A\subset Y$ of Lebesgue measure zero. 
Then for each fixed $n\geq 0$ and $m\geq 1$,
 \begin{equation}\label{eq:abscontmartwhy0}
 \begin{split}
  \mathbb E \mathbf 1_{f^m_n \in A} &= \int_{\Omega}\int_{\mathbb B_Y} \mathbf 
1_{\tilde f^{\eps}_n + \frac {1}{m}\langle b,\tilde f^{\eps}_n\rangle b \in 
A}\ud\hat{\mathbb P}(b)\ud\mathbb P\\
  &= \int_{\Omega}\int_{\mathbb B_Y} \mathbf 1_{\frac {1}{m}\langle b,\tilde 
f^{\eps}_n\rangle b \in A-\tilde f^{\eps}_n}\ud\hat{\mathbb P}(b)\ud\mathbb P,
 \end{split}
 \end{equation}
where $F-y$ is a translation of a set $F\subset Y$ by a vector $y\in Y$. For 
each fixed $y\in Y\setminus \{0\}$ the distribution of a $Y$-valued random 
variable $b\mapsto \langle b,y\rangle b$ is absolutely continuous with respect 
to $\lambda_Y$. Since $\hat{\mathbb P}(A-y)=0$ for each $y\in Y\setminus \{0\}$, 
one has
\begin{equation}\label{eq:abscontmartwhy1}
 \int_{\mathbb B_Y} \mathbf 1_{\frac {1}{m}\langle b,y\rangle b \in 
A-y}\ud\hat{\mathbb P}(b) = 0.
\end{equation}
Recall that $\mathbb P\{\tilde f^{\eps}_n =0\} =0$, therefore due to 
\eqref{eq:abscontmartwhy1} a.s.
\[
 \int_{\mathbb B_Y} \mathbf 1_{\frac {1}{m}\langle b,\tilde f^{\eps}_n\rangle b 
\in A-\tilde f^{\eps}_n}\ud\hat{\mathbb P}(b)=0.
\]
Consequently the last double integral in \eqref{eq:abscontmartwhy0} vanishes. 
The same works for $g^m$.

Now to construct such a sequence for $((f_n, 0))_{n\geq 0}$ and $((g_n, 
0))_{n\geq 0}$ one needs to construct it for different $\eps$ and take an 
appropriate subsequence.
\end{proof}

\begin{proof}[Proof of Theorem \ref{thm:forseq}]
 The ``if'' part is obvious thanks to the definition of a UMD Banach space. Let 
us prove the ``only if'' part. As in the proof of the lemma above, 
without loss of generality suppose that $X$ is separable and that the set 
$\bigcup_n(\{f_n=0\}\cup\{g_n=0\})$ is of $\mathbb P$-measure $0$. If it does 
not hold, we consider $Y:= X\oplus \mathbb R$ instead of $X$ with the norm of 
$(x,r)\in Y$ given by $\|(x,r)\|_Y = (\|x\|_X^p+|r|^p)^{1/p}$. Notice that then 
$\beta_{p,Y} = \beta_{p,X}$. We can suppose that $a_0$ is nonzero a.s., so we 
consider $(f^{\eps}_n)_{n\geq 0}:= (f_n \oplus \eps)_{n\geq 0}$ and 
$(g^{\eps}_n)_{n\geq 0}:= (g_n \oplus \eps a_0)_{n\geq 0}$ with $\eps>0$, and 
let $\eps$ go to zero. 
 
 One can also restrict to a finite dimensional case. Indeed, since $X$ is a 
separable reflexive space, $X^*$ is separable as well. Let $(Y_m)_{m\geq 1}$ be 
an increasing sequence of finite-dimensional subspaces of $X^*$ such that 
$\overline{\bigcup_mY_m}=X^*$ and $\|\cdot\|_{Y_m} = \|\cdot\|_{X^*}$ for each 
$m\geq 1$. Then for each fixed $m\geq 1$ there exists a linear operator 
$P_m:X\to Y_m^*$ of norm $1$ defined as follows: $\langle P_mx, y\rangle = 
\langle x,y\rangle$ for each $x\in X, y\in Y_m$. Then since $Y_m$ is a closed 
subspace of $X^*$, \cite[Proposition 4.33]{HNVW1} yields $\beta_{p',Y_m}\leq 
\beta_{p',X^*}$, consequently again by \cite[Proposition 4.33]{HNVW1} 
$\beta_{p,Y_m^*}\leq \beta_{p,X^{**}} = \beta_{p,X}$. So if we prove the finite 
dimensional version, then
 $$
   \mathbb E \|P_mg_n\|^p\leq \beta^p_{p,X} \mathbb E\|P_mf_n\|^p,\;\;\; n\geq 
0,
 $$
 for each $m\geq 1$, and due to the fact that $\|P_mx\|_{Y_m^*} \nearrow 
\|x\|_{X}$ for each $x\in X$ as $m \to \infty$, we would obtain 
\eqref{eq:forpnorm} in the general case.
 
 Let $\beta$ be the UMD constant of $X$, and let $U,V:X\times X\to \mathbb R$ be 
as defined in Theorem \ref{thm:Burkholder} and in \eqref{eq:defofV} 
respectively, $(a_n)_{n\geq 0}$ be as defined in Proposition 
\ref{thm:discweakdiffsub}. By Lemma \ref{lemma:abscontmart} we can suppose that 
$f_n$ and $g_n$ have distributions which are absolutely continuous with respect 
to the Lebesgue measure. Then
 \begin{align}\label{eq:verybigthing}
  \mathbb E (\|g_n\|^p-\beta \|f_n\|^p)&\stackrel{(i)}{\leq} \mathbb E 
U(f_n,g_n) = \mathbb E U(f_{n-1}+df_n,g_{n-1}+a_ndf_n)\nonumber\\
  &\stackrel{(ii)}=\mathbb E V\bigl( 
{g_{n\!-\!1}\!+\!f_{n\!-\!1}}\!+\!{(a_n\!+\!1)df_n},{g_{n\!-\!1}\!-\!f_{n\!-\!1}
}\!+\!{(a_n\!-\!1)df_n}\bigr)\nonumber\\
  &\stackrel{(iii)}{\leq} \mathbb E V\bigl( 
{g_{n-1}+f_{n-1}},{g_{n-1}-f_{n-1}}\bigr)\nonumber\\
  &\quad+ \mathbb E \bigl\langle \partial_x V\bigl({g_{n-1}+f_{n-1}}, 
{g_{n-1}-f_{n-1}}\bigr),{(a_n+1)df_n}\bigr\rangle\\
  &\quad+\mathbb E \bigl\langle \partial_y 
V\bigl({g_{n-1}+f_{n-1}},{g_{n-1}-f_{n-1}}\bigr), 
{(a_n-1)df_n}\bigr\rangle\nonumber\\
  &\stackrel{(iv)}{=}\mathbb E 
V\bigl({g_{n-1}+f_{n-1}},{g_{n-1}-f_{n-1}}\bigr)\nonumber\\
  &\stackrel{(v)}=\mathbb E U(f_{n-1},g_{n-1}).\nonumber
 \end{align}
Here $(i)$ and $(iii)$ hold by Theorem \ref{thm:Burkholder} and 
\eqref{eq:dirderivativeV} respectively, $(ii)$ and $(v)$ follow from the 
definition of $V$. Let us prove $(iv)$. We will show that
\begin{equation}\label{eq:condexpextVU}
  \mathbb E \bigl\langle \partial_x 
V\bigl({g_{n-1}+f_{n-1}},{g_{n-1}-f_{n-1}}\bigr), {(a_n+1)df_n}\bigr\rangle = 0.
\end{equation}
Since both $f_n$ and $a_n f_n$ are martingale differences, ${(a_n+1)df_n}$ is a 
martingale difference as well. Therefore $\mathbb E\bigl({(a_n-1)df_n}|\mathcal 
F_{n-1}\bigr)=0$. Note that according to Lemma~\ref{lem:V_xV_yarenice}~a.s.
\[
 \|\partial_x 
V\bigl({g_{n\!-\!1}\!+\!f_{n\!-\!1}},{g_{n\!-\!1}\!-\!f_{n\!-\!1}}\bigr)\| 
\lesssim_{V} \|f_n\|^{p-1} + \|g_n\|^{p-1}.
\]
Therefore by the H\"older inequality $\bigl\langle \partial_x 
V\bigl({g_{n\!-\!1}\!+\!f_{n\!-\!1}},{g_{n\!-\!1}\!-\!f_{n\!-\!1}}\bigr), 
{(a_n\!+\!1)df_n}\bigr\rangle$ is integrable. Since $\partial_x 
V\bigl({g_{n\!-\!1}\!+\!f_{n\!-\!1}},{g_{n\!-\!1}\!-\!f_{n\!-\!1}}\bigr)$ is 
$\mathcal F_{n-1}$-measurable,
\begin{align*}
  &\mathbb E\Bigl(\bigl\langle \partial_x V\bigl({g_{n-1}+f_{n-1}}, 
{g_{n-1}-f_{n-1}}\bigr), {(a_n+1)df_n}\bigr\rangle \big|\mathcal 
F_{n-1}\Bigr)\\\
  &=\Bigl\langle  \partial_x V\bigl({g_{n-1}+f_{n-1}}, {g_{n-1}-f_{n-1}}\bigr), 
\mathbb E\bigl( {(a_n+1)df_n} \big|\mathcal F_{n-1}\bigr)\Bigr\rangle\\
  &=\bigl\langle  \partial_x V\bigl({g_{n-1}+f_{n-1}}, {g_{n-1}-f_{n-1}}\bigr), 
0\bigr\rangle = 0,
\end{align*}
so \eqref{eq:condexpextVU} holds. By the same reason
$$
\mathbb E \bigl\langle \partial_y 
V\bigl({g_{n-1}+f_{n-1}},{g_{n-1}-f_{n-1}}\bigr), {(a_n-1)df_n}\bigr\rangle=0,
$$
and $(iv)$ follows.

Notice that thanks to Remark \ref{rem:propsofU} $\mathbb E(f_0, g_0)\leq 0$. 
Therefore from the inequality \eqref{eq:verybigthing} by an induction argument 
we get
\begin{align*}
  \mathbb E (\|g_n\|^p-\beta^p \|f_n\|^p) &\leq \mathbb E U(f_{n},g_{n})
  \leq \mathbb E U(f_{n-1},g_{n-1})\leq \ldots\leq \mathbb E U(f_0,g_0)\leq 0.
\end{align*}
This terminates the proof.
\end{proof}

\subsection{Continuous time case}\label{subsec:contcase}
Now we turn to continuous time martingales. Let $(\Omega, \mathcal F, \mathbb 
P)$ be a probability space with a filtration $\mathbb F = (\mathcal F_t)_{t\geq 
0}$ that satisfies the usual conditions. 

\begin{defi}\label{def:pureludisc}
 Let $M:\mathbb R_+\times \Omega \to \mathbb R$ be a local martingale. Then $M$ 
is called {\em purely discontinuous} if $[M]$ is a pure jump processes (i.e.\ 
$[M]$ has a version that is a constant a.s.\ in time). Let $X$ be a Banach 
space, $M:\mathbb R_+\times \Omega \to X$ be a local martingale. Then $M$ is 
called {\em purely discontinuous} if for each $x^* \in X^*$ a local martingale 
$\langle M, x^*\rangle$ is purely discontinuous.
\end{defi}

The reader can find more on purely discontinuous martingales in \cite{JS,Kal}. 

\begin{defi}
 Let $M,N:\mathbb R_+ \times \Omega \to X$ be local martingales. Then we say 
that $N$ is {\em weakly differentially subordinated} to $M$ if for each $x^* \in 
X^*$ one has that $[\langle M,x^*\rangle]-[\langle N,x^*\rangle]$ is an a.s.\ 
nondecreasing function and $|\langle N_0,x^*\rangle|\leq |\langle 
M_0,x^*\rangle|$~a.s.
\end{defi}

 The following theorem is a natural extension of Proposition 
\ref{thm:discweakdiffsub}.

\begin{theorem}\label{thm:weakdiffsubproc}
Let $X$ be a Banach space. Then $X$ is a UMD space if and only if for some (equivalently, for all) $1<p<\infty$ there exists 
$\beta >0$ such that for each purely discontinuous $X$-valued local martingales 
$M,N:\mathbb R_+\times \Omega \to X$ such that $N$ is weakly differentially 
subordinated to $M$ one has
  \begin{equation}\label{eq:thmforpurejump}
  \mathbb E \|N_t\|^p\leq \beta^p\mathbb E \|M_t\|^p.
 \end{equation}
 If this is the case then the smallest admissible $\beta$ equals the UMD 
constant $\beta_{p,X}$.
\end{theorem}

\begin{lemma}\label{lemma:trulyabscontmart}
 Let $X$ be a finite dimensional Banach space, $1<p<\infty$, $M, N:\mathbb R_+ 
\times \Omega \to X$ be local martingales on a probability space $(\Omega, 
\mathcal F, \mathbb P)$ with a filtration $\mathbb F = (\mathcal F_t)_{t\geq 0}$ 
such that $N$ is weakly differentially subordinated to $M$. Let $Y=X\oplus 
\mathbb R$ be a Banach space such that $\|(x,r)\|_Y= (\|x\|_X^p + |r|^p)^{\frac 
1p}$ for each $x\in X$, $r\in \mathbb R$. Then there exist two sequences 
$(M^m)_{m\geq 1}$ and $(N^m)_{m\geq 1}$ of $Y$-valued martingales on an enlarged 
probability space $(\overline{\Omega}, \overline{\mathcal F}, \overline{\mathbb 
P})$ with an enlarged filtration $\overline{\mathbb F} = (\overline{\mathcal 
F}_t)_{t\geq 0}$ such that
 \begin{enumerate}
  \item $M^m_t$, $N^m_t$ have absolutely continuous distributions with respect 
to the Lebesgue measure on $Y$ for each $m\geq 1$ and $t\geq 0$;
  \item $M^m_t \to (M_t,0)$, $N^m_t \to (N_t,0)$ pointwise as $m\to \infty$ for 
each $t\geq 0$;
  \item if for some $t\geq 0$ $\mathbb E\|M_t\|^p<\infty$, then for each $m\geq 
1$ one has that $\mathbb E\|M^m_t\|^p<\infty$ and $\mathbb 
E\|M^m_t-(M_t,0)\|^p\to 0$ as $m\to \infty$;
  \item if for some $t\geq 0$ $\mathbb E\|N_t\|^p<\infty$, then for each $m\geq 
1$ one has that $\mathbb E\|N^m_t\|^p\!~\!<\!~\!\infty$ and $\mathbb 
E\|N^m_t-(N_t,0)\|^p\to 0$ as $m\to \infty$;
  \item for each $m\geq 1$ we have that $N^m$ is weakly differentially 
subordinated to~$M^m$.
 \end{enumerate}
\end{lemma}

\begin{proof}
 The proof in essentially the same as one of Lemma \ref{lemma:abscontmart}.
\end{proof}

\begin{proof}[Proof of Theorem \ref{thm:weakdiffsubproc}]
We use a modification of the argument in \cite[Theorem 1]{Wang}, where the 
Hilbert space case was considered. Thanks to the same methods as were applied in 
the beginning of the proof of Theorem \ref{thm:forseq} and using Lemma 
\ref{lemma:trulyabscontmart} instead of Lemma \ref{lemma:abscontmart}, one can 
suppose that $X$ is finite-dimensional and $M_t$ and $N_t$ are nonzero a.s.\ for 
each $t\geq 0$. We know that $\mathbb E U(M_t,N_t)\geq \mathbb 
E(\|N_t\|^p-\beta^p \|M_t\|^p)$ for each $t\geq 0$. On the other hand, thanks to 
the fact that $[\langle M, x^*\rangle]$ and $[\langle N, x^*\rangle]$ are pure 
jump for each $x^* \in X^*$ and the finite-dimensional version of It\^o formula 
\cite[Theorem 26.7]{Kal}, one has
 \begin{equation}\label{eq:EU(M,N)}
  \begin{split}
     \mathbb E U(M_t,N_t) = \mathbb E U(M_0,N_0) &+ \mathbb E \int_{0}^t \langle 
\partial_x U(M_{s-},N_{s-}),\ud M_s\rangle\\
  &+\mathbb E \int_{0}^t \langle \partial_y U(M_{s-},N_{s-}),\ud N_s\rangle + 
\mathbb E I,
  \end{split}
 \end{equation}
where
$$
I = \sum_{0<s\leq t}[\Delta U(M_s,N_s)-\langle \partial_x 
U(M_{s-},N_{s-}),\Delta M_s\rangle - \langle \partial_y U(M_{s-},N_{s-}),\Delta 
N_s\rangle].
$$
Note that since a.s.
$$
\Delta |\langle N,x^*\rangle|^2= \Delta [\langle N,x^*\rangle]\leq \Delta 
[\langle M,x^*\rangle]=\Delta |\langle M,x^*\rangle|^2
$$ 
for each $x^* \in X^*$, one has that thanks to Lemma 
\ref{lemma:ajustformeasfunc} for each $s\geq 0$, for a.e.\ $\omega \in \Omega$ 
there exists $a_s(\omega)$ such that $|a_s(\omega)|\leq 1$ and $\Delta 
N_s(\omega) = a_s(\omega)\Delta M_s(\omega)$. Therefore for each $s\geq 0$ by 
\eqref{eq:dirderivativeV} $\mathbb P$-a.s.
\begin{align*}
 &\Delta U(M_s,N_s)-\langle \partial_x U(M_{s-},N_{s-}),\Delta M_s\rangle - 
\langle \partial_y U(M_{s-},N_{s-}),\Delta N_s\rangle\\
 &= V(M_{s-}+N_{s-}+(a_s+1)\Delta M_s,N_{s-}-M_{s-} + (a_s-1)\Delta M_s)\\ 
 &\quad- V(M_{s-}+N_{s-},N_{s-}-M_{s-})\\ 
 &\quad-\langle \partial_x V(M_{s-}+N_{s-},N_{s-}-M_{s-}),(a_s+1)\Delta 
M_s\rangle\\
 &\quad- \langle \partial_y V(M_{s-}+N_{s-},N_{s-}-M_{s-}),(a_s-1)\Delta 
M_s\rangle \leq 0,
\end{align*}
so $I\leq 0$ a.s., and $\mathbb E I \leq 0$. Also 
\begin{align*}
 \int_{0}^t \langle \partial_x U(M_{s-},N_{s-}),\ud M_s\rangle &+ \int_{0}^t 
\langle \partial_y U(M_{s-},N_{s-}),\ud N_s\rangle\\
 &=\int_{0}^t \langle \partial_x V(M_{s-}+N_{s-},N_{s-}-M_{s-}),\ud 
(M_s+N_s)\rangle \\
 &+ \int_{0}^t \langle  \partial_y V(M_{s-}+N_{s-},N_{s-}-M_{s-}),\ud 
(N_s-M_s)\rangle,
\end{align*}
so by Lemma \ref{lemma:stochintmoment} and Lemma \ref{lem:V_xV_yarenice} 
it is a martingale that starts at zero, and therefore its expectation is zero as 
well. Consequently, thanks to \eqref{eq:ineqonU}, \eqref{eq:EU(M,N)} and Remark 
\ref{rem:propsofU}, 
$$
\mathbb E \|N_t\|^p - \beta_{p,X}^p\mathbb E \|M_t\|^p \leq \mathbb E U(M_t,N_t) 
\leq \mathbb E U(M_0,N_0) \leq 0,
$$
and therefore \eqref{eq:thmforpurejump} holds.
\end{proof}

As one can see, in our proof we did not need the second order terms of the It\^o 
formula thanks to the nature of the quadratic variation of a purely 
discontinuous process. Nevertheless, Theorem \ref{thm:weakdiffsubproc} holds for 
arbitrary martingales $M$ and $N$, but with worse estimates (see 
\cite{Y17MartDec}). The connection of Theorem \ref{thm:weakdiffsubproc} for 
continuous martingales with the Hilbert transform will be discussed in Section 
\ref{sec:Hilbtrans}.

\section{Fourier multipliers}

In \cite{BB} and \cite{BBB} the authors exploited the differential subordination 
property to show boundedness of certain Fourier multipliers in $\mathcal 
L(L^p(\mathbb R^d))$. It turned out that it is sufficient to use the weak 
differential subordination property to obtain the same assertions, but 
in the vector-valued situation.

\subsection{Basic definitions and the main theorem}
Let $d\geq 1$ be a natural number. Recall that $\mathcal S(\mathbb R^d)$ is a 
space of Schwartz functions on $\mathbb R^d$. For a Banach space $X$ with a 
scalar field $\mathbb C$ we define $\mathcal S(\mathbb R^d)\otimes X$ as the space 
of all functions $f:\mathbb R^d \to X$ of the form $f=\sum_{k=1}^K f_k \otimes 
x_k$, where $K\geq 1$, $f_1,\ldots,f_K \in \mathcal S(\mathbb R^d)$, and 
$x_1,\ldots ,x_K \in X$. Notice that for each $1\leq p<\infty$ the space 
$\mathcal S(\mathbb R^d)\otimes X$ is dense in $L^p(\mathbb R^d; X)$.

We define the {\it Fourier transform} $\mathcal F$ and the {\it inverse Fourier 
transform} $\mathcal F^{-1}$ on $\mathcal S(\mathbb R^d)$ as follows:
\begin{align*}
 \mathcal F (f)(t) &= \frac{1}{(2\pi)^{\frac d2}}\int_{\mathbb R^d} e^{-i\langle 
t,u\rangle}f(u)\ud u,\;\;\; f\in \mathcal S(\mathbb R^d), t\in \mathbb R^d,\\
  \mathcal F^{-1} (f)(t) &= \frac{1}{(2\pi)^{\frac d2}}\int_{\mathbb R^d} 
e^{i\langle t,u\rangle}f(u)\ud u,\;\;\; f\in \mathcal S(\mathbb R^d), t\in 
\mathbb R^d.
\end{align*}
It is well-known that for any $f\in \mathcal S(\mathbb R^d)$ we have $\mathcal 
F(f), \mathcal F^{-1}(f) \in \mathcal S(\mathbb R^d)$, and $\mathcal 
F^{-1}(\mathcal F(f)) = f$. The reader can find more details on the Fourier 
transform in \cite{Graf}.

Let $m:\mathbb R^d\to \mathbb C$ be measurable and bounded. Then we can 
define a linear operator $T_m$ on $\mathcal S(\mathbb R^d)\otimes X$ as follows:
\begin{equation}\label{eq:multSotimesX}
  T_m (f\otimes x) = \mathcal F^{-1}(m\mathcal F(f)) \cdot x,\;\;\;\; f\in 
\mathcal S(\mathbb R^d), x\in X.
\end{equation}
 The operator $T_m$ is called a {\it Fourier multiplier}, while the function $m$ 
is called the {\it symbol} of $T_m$. If $X$ is finite-dimensional then $T_m$ can 
be extended to a bounded linear operator on $L^2(\mathbb R^d; X)$. The question 
is usually whether one can extend $T_m$ to a bounded operator on $L^p(\mathbb 
R^d;X)$ for a general $1<p<\infty$ and a given $X$. Here the answer will be 
given for $m$ of quite a special form and $X$ with the UMD property.

Let $V$ be a L\'evy measure on $\mathbb R^d$, that is $V(\{0\})=0$, $V \neq 0$ 
and
$$
\int_{\mathbb R^d}(|x|^2\wedge 1) V(\dd x) <\infty.
$$
Let $\phi \in L^{\infty}( \mathbb R^d ; \mathbb C)$ be such that 
$\|\phi\|_{L^{\infty}(\mathbb R^d;\mathbb C)} \leq 1$. Also let $\mu\geq 0$ be a 
finite Borel measure on the unit sphere $S^{d-1} \subset \mathbb R^d$, and 
$\psi\in L^{\infty}( S^{d-1}; \mathbb C)$ satisfies
$\|\psi\|_{L^{\infty}(S^{d-1};\mathbb C)}\!~\!\leq\!~\!1$.

In the sequel we set $\frac {a}{0} = 0$ for each $a \in \mathbb C$. The 
following result extends \cite[Theorem 1.1]{BBB} to the UMD Banach space 
setting.

\begin{theorem}\label{thm:mainmult}
 Let $X$ be a UMD Banach space. Then the Fourier multiplier $T_m$ with a symbol
 \begin{equation}\label{eq:defofm(xi)}
  m(\xi) = \frac{\int_{\mathbb R^d} (1-\cos \xi \cdot z)\phi(z)V(\dd z) + 
\frac 12\int_{S^{d-1}}(\xi \cdot \theta)^{2}\psi(\theta)\mu(\dd\theta)}{\int_{\mathbb 
R^d} (1-\cos \xi \cdot z)V(\dd z)+ \frac 12\int_{S^{d-1}}(\xi \cdot 
\theta)^{2}\mu(\dd\theta)},\;\; \xi \in \mathbb R^d,
 \end{equation}
 has a bounded extension on $L^p(\mathbb R^d; X)$ for $1<p<\infty$. Moreover, 
then for each $f\in L^p(\mathbb R^d; X)$
 \begin{equation}\label{eq:multthm}
   \|T_m f\|_{p}\leq \beta_{p,X}\|f\|_p.
 \end{equation}
\end{theorem}

\begin{remark}
 The coefficient $\frac 12$ in both numerator and denominator of \eqref{thm:mainmult}, even though it looks wired and useless (because one can always transform $\mu$ to $2\mu$), exists because of the strong connection with the L\'evy--Khintchin representation of L\'evy processes (see e.g.\ \cite[Part~4.1]{Ban10}).
\end{remark}

The proof is a modification of the arguments given in \cite{BBB} and \cite{BB}, 
but instead of real-valued process we will work with processes that take their 
values in a finite dimensional space. For the convenience of the reader the 
proof will be given in the same form and with the same notations as the original 
one. However, we will need to justify here some steps, so we cannot just skip 
the proof. First of all as that was done in \cite{BBB}, we reduce to the case of 
symmetric $V$ and $\mu = 0$, and proceed as in the proof of \cite[Theorem 
1]{BB}.

In the rest of the section we may assume that $X$ is finite dimensional, since 
it is sufficient to show \eqref{eq:multthm} for all $f$ with values in $X_0$ for 
each finite dimensional subspace $X_0$ of $X$.

Let $\nu$ be a positive finite symmetric measure on $\mathbb R^d$, 
$\widetilde{\nu} = \nu/|\nu|$. Let $T_i$ and $Z_i$, $i = \pm 1, \pm 2,\pm 3, 
\ldots$, be a family of independent random variables, such that each $T_i$ is 
exponentially distributed with parameter $|\nu|$ (i.e.\ $\mathbb E T_i = 
1/|\nu|$), and each $Z_i$ has $\widetilde{\nu}$ as a distribution. Let $S_i = 
T_1+\cdots +T_i$ for a positive $i$ and $S_i = -(T_{-1}+\cdots +T_i)$ for a 
negative $i$. For each $-\infty < s<t<\infty$ we define $X_{s,t}:= 
\sum_{s<S_i\leq t}Z_i$ and $X_{s,t-}:= \sum_{s<S_i< t}Z_i$. Note that $\mathcal 
N(B) = \#\{i:(S_i, Z_i)\in B\}$ defines a Poisson measure on $\mathbb R\times 
\mathbb R^d$ with the intensity measure $\lambda \otimes \nu$, and $X_{s,t} = 
\int_{s<v\leq t} x \mathcal N(\dd v, \ud \nu)$ (see e.g.\ \cite{Sato}). Let 
$N(s,t) = \mathcal N((s,t] \times \mathbb R^d)$ be the number of {\it signals} 
$S_i$ such that $s<S_i \leq t$. The following Lemmas 
\ref{lem:lemma1inBB}-\ref{lem:lemma5inBB} are multidimensional 
versions of \cite[Lemma 1--5]{BB}, which can be proven in the same way as in the 
scalar case.

\begin{lemma}\label{lem:lemma1inBB}
 Let $f: \mathbb R\times\mathbb R^d\times\mathbb R^d \to X$ be Borel measurable 
and be either nonnegative or bounded, and let $s\leq t$. Then
 $$
 \mathbb E \sum_{s<S_i\leq t}F(S_i, X_{s,S_i-, X_{s,S_i}}) = \mathbb E \int_s^t 
\int_{\mathbb R^d}F(v,X_{s,v-, X_{s,v-}+z})\nu(\dd z)\ud v.
 $$
\end{lemma}
 We will consider the following filtration:
 $$
 \mathbb F = \{\mathcal F_t\}_{t\in \mathbb R} = \{\sigma\{X_{s,t}:s\leq 
t\}\}_{t\in \mathbb R}.
$$
Recall that for measures $\nu_1$ and $\nu_2$ on $\mathbb R^d$ the expression 
$\nu_1 \ast \nu_2$ means the {\it convolution} of these measures (we refer the 
reader \cite[Chapter 3.9]{Bog} for the details). Also for each $n\geq 1$ we 
define $\nu_1^{*n}:= \underbrace{\nu_1\ast \cdots \ast\nu_1}_{n \text{ times}}.$ 
For each $t\in \mathbb R$ define
\begin{equation*}
 p_t = e^{*t(\nu-|\nu|\delta_0)} = 
\sum_{n=0}^{\infty}\frac{t^n}{n!}(\nu-|\nu|\delta_0)^{*n}=e^{-t|\nu|}\sum_{n=0}^
{\infty} \frac{t^n}{n!}\nu^{*n}.
\end{equation*}
The series converges in the norm of absolute variation of measures. As in 
\cite[(18)]{BB} and \cite[(3.9)]{BBB} $p_t$ is symmetric, and 
$$
\frac{\partial}{\partial t}p_t = (\nu - |\nu|\delta_0)* p_t,\;\;\; t\in \mathbb 
R.
$$
Also $p_{t_1+t_2} = p_{t_1}*p_{t_2}$ for each $t_1, t_2\in \mathbb R$. In fact 
for all $t\leq u$ the measure $p_{u-t}$ is the distribution of $X_{t,u}$ and 
$X_{t,u-}$. Put
$$
\Psi(\xi) = \int_{\mathbb R^d} (e^{i\xi \cdot z}-1)\nu(\dd z),\;\;\; \xi\in 
\mathbb R^d.
$$
Thanks to the symmetry of $\nu$ one has as well that
$$
\Psi(\xi) = \int_{\mathbb R^d} (\cos\xi\cdot z - 1)\nu(\dd z) = \Psi(-\xi)\leq 
0.
$$
Therefore $\Psi$ is bounded on $\mathbb R^d$, and due to \cite[(3.12)]{BBB} 
we have that the characteristic function of $p_t$ is of the following form: 
$$
\hat p_t(\xi) = e^{t\Psi(\xi)},\;\;\;\xi \in \mathbb R^d.
$$
(The reader can find more on characteristic functions in \cite[Chapter 
3.8]{Bog}.)

Let $g \in L^{\infty}(\mathbb R^d;X)$. Then for $x\in \mathbb R^d$, $t\leq u$, 
we define the {\it parabolic extension} of $g$ by
\begin{equation*}
 P_{t,u}g(x) := \int_{\mathbb R^d}g(x+y)p_{u-t}(\dd y) = g*p_{u-t}(x)=\mathbb 
Eg(x+X_{t,u}).
\end{equation*}
For $s\leq t\leq u$ we define the {\it parabolic martingale} by
\begin{equation*}
 G_t = G_t(x;s,u;g):= P_{t,u}g(x+X_{s,t}).
\end{equation*}
\begin{lemma}
 We have that $G_t$ is a bounded $\mathbb F$-martingale.
\end{lemma}

Let $\phi\in L^{\infty}(\mathbb R^d;\mathbb C)$ be symmetric. For each $x\in 
\mathbb R^d$, $s\leq t\leq u$, and $f\in C_c(\mathbb R^d;X)$ we define $F_t$ as 
follows:
\begin{align*}
 F_t &= F_t(x;s,u;f,\phi) := \\
 &\sum_{s<S_i\leq t}[P_{S_i,u}f(x+X_{s,S_i})- 
P_{S_i,u}f(x+X_{s,S_i-})]\phi(X_{s, S_i}-X_{s, S_i-})\\
 &- \int_s^t \int_{\mathbb 
R^d}[P_{v,u}f(x+X_{s,v-}+z)-P_{u,v}f(x+X_{s,v-})]\phi(z)\nu(\dd z)\ud v.
\end{align*}
\begin{lemma}
 We have that $F_t = F_t(x;s,u;f,\phi)$ is an $\mathbb F$-martingale for $t\in [s,u]$. 
Moreover, $\mathbb E\|F_t\|^p <\infty$ for each $p>0$.
\end{lemma}
\begin{lemma}\label{lem:lemma5inBB}
 $G_t(x;s,u;g) = F_t(x;s,u;g,1)+P_{s,u}g(x)$.
\end{lemma}
Analogously to \cite[(21)-(22)]{BB} one has that for each $x^* \in X^*$ the 
quadratic variations of $\langle F_t(x;s,u;f,\phi), x^*\rangle$ and $\langle 
G_t(x;s,u;g), x^*\rangle$ satisfy the following a.s.\ identities,
\begin{equation*}
 [\langle F,x^* \rangle]_t = \sum_{s<S_i\leq t}\Bigl(\langle 
P_{S_i,u}f(x+X_{s,S_i})-P_{S_i,u}f(x+X_{s,S_i-}), 
x^*\rangle\Bigr)^2\phi^2(\Delta X_{s, S_i}),
\end{equation*}
\begin{equation*}
 [\langle G,x^* \rangle]_t = |\langle P_{s,u} g(x), 
x^*\rangle|^2+\sum_{s<S_i\leq t}\Bigl(\langle 
P_{S_i,u}g(x+X_{s,S_i})-P_{S_i,u}g(x+X_{s,S_i-}), x^*\rangle\Bigr)^2.
\end{equation*}
It follows that for each $f\in C_c(\mathbb R^d;X)$, 
$(F_t(x;s,u;f,\phi))_{t\in[s,u]}$ is weakly differentially subordinated to 
$(G_{t}(x;s,u;f))_{t\in[s,u]}$ and by Theorem \ref{thm:weakdiffsubproc} one has 
for each $t\in [s,u]$
$$
\mathbb E\|F_t(x;s,u;f,\phi)\|^p \leq \beta_{p,X}^p \mathbb E\|G_t(x;s,u;f)\|^p.
$$
Note that $G_u(x;s,u;f) = f(x+X_{s,u})$, so
\begin{equation}\label{eq:fourcorfrommainthm}
\begin{split}
 \int_{\mathbb R^d}\mathbb E\|F_u(x;s,u;f,\phi)\|^p\ud x &\leq 
\beta_{p,X}^p\int_{\mathbb R^d}\mathbb E \|f(x+X_{s,u})\|^p\ud x \\
 &= \beta_{p,X}^p \|f\|^p_{L^p(\mathbb R^d;X)}.
\end{split}
\end{equation}
 Let $p'$ be such that $\frac 1p + \frac{1}{p'}=1$. Consider the linear 
functional on $L^{p'}(\mathbb R^d;X^*)$:
$$
L^{p'}(\mathbb R^d;X^*) \ni g\mapsto \int_{\mathbb R^d} \mathbb E \langle 
F_u(x;s,u;f,\phi),g(x+X_{s,u}) \rangle\ud x. 
$$
Then by H\"older's inequality and \eqref{eq:fourcorfrommainthm} one has
$$
\int_{\mathbb R^d}\mathbb E |\langle F_u(x;s,u;f,\phi),g(x+X_{s,u})\rangle|\ud x 
\leq \beta_{p,X}\|f\|_{L^p(\mathbb R^d;X)}\|g\|_{L^{p'}(\mathbb R^d;X^*)}.
$$
By Theorem 1.3.10 and Theorem 1.3.21 in \cite{HNVW1}, $(L^{p'}(\mathbb 
R^d;X^*))^* = L^p(\mathbb R^d;X)$, so there exists $h\in L^p(\mathbb 
R^d;X)$ such that for each $g\in L^{p'}(\mathbb R^d;X^*)$
$$
\int_{\mathbb R^d} \mathbb E \langle F_u(x;s,u;f,\phi),g(x+X_{s,u}) \rangle\ud x 
= \int_{\mathbb R^d}\langle h(x),g(x)\rangle\ud x,
$$
and
\begin{equation}\label{eq;ineqfornorm}
 \|h\|_{L^p(\mathbb R^d;X)}\leq \beta_{p,X}\|f\|_{L^p(\mathbb R^d;X)}.
\end{equation}
In particular, since $X$ is finite dimensional
\begin{equation}\label{eq:fourscalarcase}
 \int_{\mathbb R^d} \mathbb E  F_u(x;s,u;f,\phi)g(x+X_{s,u})\ud x = 
\int_{\mathbb R^d} h(x)g(x)\ud x,\;\;\; g\in L^{p'}(\mathbb R^d).
\end{equation}
For each $s<0$ define $m_s:\mathbb R^d \to \mathbb C$ as follows
$$
m_s(\xi) =
\begin{cases}
 \Bigl(1-e^{2|s|\Psi(\xi)}\Bigr)\frac{1}{\Psi(\xi)}\int_{\mathbb R^d}(e^{i\xi\cdot 
z}-1)\phi(z)\nu(\dd z),\;\;\; &\Psi(\xi)\neq 0,\\
 0,\;\;\;&\Psi(\xi)= 0.
\end{cases}
$$

Let $u=0$. Then analogously to \cite[(32)]{BB}, by 
\eqref{eq:fourscalarcase} one obtains
$$
\mathcal F (h)(\xi) = m_s(\xi) \mathcal F (f)(\xi),\;\;\; \xi \in \mathbb R^d.
$$
Let $T_{m_s}$ be the Fourier multiplier on $L^2(\mathbb R^d;X)$ with symbol 
$m_s$ (that is bounded by 1). By \eqref{eq;ineqfornorm} one obtains that 
$T_{m_s}$ extends uniquely to a bounded operator on $L^p(\mathbb R^d; X)$ with 
$\|T_{m_s}\|_{\mathcal L(L^p(\mathbb R^d;X))}\leq\beta_{p,X}$. Let $T_m$ be the 
multiplier on $L^2(\mathbb R^d;X)$ with the symbol $m$ given by
$$
m(\xi) =
\begin{cases}
 \frac{1}{\Psi(\xi)}\int_{\mathbb R^d}(e^{i\xi\cdot z}-1)\phi(z)\nu(\dd z),\;\;\; 
&\Psi(\xi)\neq 0,\\
 0,\;\;\;&\Psi(\xi)= 0.
\end{cases}
$$
Note that $m$ is a pointwise limit of $m_s$ as $s\to -\infty$. Also note that 
$T_{m_s} f \to T_{m} f$ in $L^2(\mathbb R^d;X)$ as $s\to -\infty$ for each $f\in 
C_c(\mathbb R^d;X)$ by Plancherel's theorem. Therefore by Fatou's lemma one has 
that for each $f\in C_c(\mathbb R^d;X)$ the following holds:
$$
\|T_{m} f\|_{L^p(\mathbb R^d;X)} \leq \varliminf_{s\to -\infty}\|T_{m_s} 
f\|_{L^p(\mathbb R^d;X)}\leq \beta_{p,X} \|f\|_{L^p(\mathbb R^d;X)},
$$
hence $T_{m}$ uniquely extends to a bounded operator on $L^p(\mathbb R^d;X)$ 
with 
$$
\|T_m\|_{\mathcal L(L^p(\mathbb R^d;X))}~\leq~ \beta_{p,X}.
$$

\subsection{Examples of Theorem \ref{thm:mainmult}}\label{sec:examples}
In this subsection $X$ is a UMD Banach space, $p\in (1,\infty)$.
The examples will be mainly the same as were given in \cite[Chapter 4]{BBB} with 
some author's remarks. Recall that we set $\frac {a}{0} = 0$ for any $a \in 
\mathbb C$.

\begin{example}
 Let $V_1, V_2$ be two nonnegative L\'evy measures on $\mathbb R^d$ such that 
$V_1\leq V_2$. Let
 \[
  m(\xi)= \frac{\int_{\mathbb R^d} (1-\cos (\xi \cdot z))V_1(\dd z) 
}{\int_{\mathbb R^d} (1-\cos (\xi \cdot z))V_2(\dd z)}, \;\;\; \xi \in \mathbb 
R^d.
 \]
Then $\|T_m\|_{\mathcal L(L^p(\mathbb R^d; X))}\leq\beta_{p,X}$.
\end{example}

\begin{example}
 Let $\mu_1, \mu_2$ be two nonnegative measures on $S^{d-1}$ such that 
$\mu_1\leq \mu_2$. Let
 \[
 m(\xi) =\frac{ \int_{S^{d-1}}(\xi \cdot 
\theta)^{2}\mu_1(\dd\theta)}{\int_{S^{d-1}}(\xi \cdot 
\theta)^{2}\mu_2(\dd\theta)}, \;\;\; \xi \in \mathbb R^d.
 \]
Then $\|T_m\|_{\mathcal L(L^p(\mathbb R^d; X))}\leq\beta_{p,X}$.
\end{example}

\begin{example}[Beurling-Ahlfors transform]
 Let $d=2$. Put $\mathbb R^2 = \mathbb C$. Then the Fourier multiplier $T_m$ 
with a symbol $m(z) = \frac{\bar z^2}{|z|^2}$, $z \in \mathbb C$, has the norm 
at most $2\beta_{p,X}$ on $L^p(\mathbb R^d; X)$. 
 This multiplier is also known as the {\it Beurling-Ahlfors transform}. It is 
well-known that $\|T_m\|_{\mathcal L(L^p(\mathbb R^2; X))}\geq \beta_{p,X}$. 
There is quite an old problem whether $\|T_m\|_{\mathcal L(L^p(\mathbb R^2; 
X))}= \beta_{p,X}$. This question was firstly posed by Iwaniec in \cite{Iw82} in 
$\mathbb C$. Nevertheless it was neither proved nor disproved even in the 
scalar-valued case. We refer the reader to \cite{Ban10} and \cite{HNVW1} for 
further details.
\end{example}

\begin{example}\label{ex:BBB4.7}
 Let $\alpha \in (0,2)$, $\mu$ be a finite positive measure on $S^{d-1}$, $\psi$ 
be a measurable function on $S^{d-1}$ such that $|\psi|\leq 1$. Let
 \[
 m(\xi) =\frac{ \int_{S^{d-1}}|(\xi \cdot 
\theta)|^{\alpha}\psi(\theta)\mu(\dd\theta)}{\int_{S^{d-1}}|(\xi \cdot 
\theta)|^{\alpha}\mu(\dd\theta)}, \;\;\; \xi \in \mathbb R^d.
 \]
Then analogously to \cite[(4.7)]{BBB}, $\|T_m\|_{\mathcal L(L^p(\mathbb R^d; 
X))}\leq\beta_{p,X}$.
\end{example}

\begin{example}[Double Riesz transform]
 Let $\alpha \in (0,2]$. Let
 \[
 m(\xi) = \frac{|\xi_1|^{\alpha}}{|\xi_1|^{\alpha}+\cdots +|\xi_d|^{\alpha}}, 
\;\;\; \xi = (\xi_1,\ldots, \xi_d) \in \mathbb R^d,
 \]
Then according to Example \ref{ex:BBB4.7}, $\|T_m\|_{\mathcal L(L^p(\mathbb R^d; 
X))}\leq\beta_{p,X}$. Note that if $\alpha = 2$, then $T_m$ is a double Riesz 
transform. (In the forthcoming paper \cite{Y18UMDAp} it is shown that the norm $\|T_m\|_{\mathcal L(L^p(\mathbb R^d; X))}$ does not depend on $\alpha$ and equals the UMD$_p^{\{0,1\}}$ constant of~$X$).
\end{example}

\begin{example}
 Let $\alpha \in [0,2]$, $d\geq 2$. Let
 \[
 m(\xi) = \frac{|\xi_1|^{\alpha}-|\xi_2|^{\alpha}}{|\xi_1|^{\alpha}+\cdots 
+|\xi_d|^{\alpha}}, \;\;\; \xi = (\xi_1,\ldots, \xi_d) \in \mathbb R^d,
 \]
Then by Example \ref{ex:BBB4.7}, $\|T_m\|_{\mathcal L(L^p(\mathbb R^d; 
X))}\leq\beta_{p,X}$. Moreover, if $d=2$, $\alpha~\in~[1,2]$, then $\max_{\xi\in 
\mathbb R^2} m(\xi) = 1$, $\min_{\xi\in \mathbb R^2} m(\xi) = -1$ and 
$m|_{S^1}\in W^{2,1}(S^1)$. Therefore due to Proposition~3.4, Proposition~3.8 
and Remark~3.9 in \cite{GM-SS} one has $\|T_m\|_{\mathcal L(L^p(\mathbb R^2;X))} 
\geq \beta_{p,X}$. This together with Theorem \ref{thm:mainmult} implies 
$\|T_m\|_{\mathcal L(L^p(\mathbb R^2;X))} = \beta_{p,X}$, which extends 
\cite[Theorem 1.1]{GM-SS}, where the same assertion was proven for $\alpha = 2$.
\end{example}

\begin{example}
 Let $\mu$ be a nonnegative Borel measure on $S^{d-1}$, $\psi \in 
L^{\infty}(S^{d-1}, \mu)$, $\|\psi\|_{\infty} \leq 1$. Let
 \[
  m(\xi)= \frac{\int_{S^{d-1}} \ln(1+(\xi\cdot \theta)^{-2})\psi(\theta)\mu(\dd 
\theta)}{\int_{S^{d-1}} \ln(1+(\xi\cdot \theta)^{-2})\mu(\dd \theta)},\;\;\;\xi 
\in \mathbb R^d.
 \]
Then $\|T_m\|_{\mathcal L(L^p(\mathbb R^d; X))}\leq\beta_{p,X}$.
\end{example}

\section{Hilbert transform and general conjecture}\label{sec:Hilbtrans}

In this section we assume that $X$ is a finite dimensional Banach space to avoid 
difficulties with stochastic integration. Many of the assertions below can be 
extended to the general UMD Banach space case by using the same techniques as in 
the proof of Theorem~\ref{thm:forseq}.

\subsection{Hilbert transform and Burkholder functions}
It turns out that the generalization of Theorem \ref{thm:weakdiffsubproc} to the 
case of continuous martingales is connected with the boundedness of the Hilbert 
transform. The Fourier multiplier $\mathcal H\in \mathcal L(L^2(\mathbb R))$ with the 
symbol $m\in L^{\infty}(\mathbb R)$ such that $m(t) = -i\; \sign(t)$, 
$t\in\mathbb R$, is called the {\it Hilbert transform}. This operator can be 
extended to a bounded operator on $L^p(\mathbb R)$, $1<p<\infty$ (see 
\cite{Riesz28} and \cite[Chapter 5.1]{HNVW1} for the details).

Let $X$ be a Banach space. Then one can extend the Hilbert transform $\mathcal 
H$ to $\mathcal S(\mathbb R)\otimes X$ in the same way as it was done in 
\eqref{eq:multSotimesX}. Denote this extension by $\mathcal H_X$.
By \cite[Lemma 2]{Bour83} and \cite[Theorem 3]{Gar85} the following holds 
true:
\begin{theorem}[Bourgain, Burkholder]
 Let $X$ be a Banach space. Then $X$ is a~UMD Banach space if and only if 
$\mathcal H_X$ can be extended to a bounded operator on $L^p(\mathbb R; X)$ for 
each $1<p<\infty$. Moreover, then
 \begin{equation}\label{eq:squaredepend}
  \sqrt {\beta_{p,X}}\leq \|\mathcal H_X\|_{\mathcal L(L^p(\mathbb R; X))} \leq 
\beta_{p,X}^2.
 \end{equation}
\end{theorem}

The proof of the right-hand side of \eqref{eq:squaredepend} is based on the 
following result.
\begin{prop}\label{thm:01-10matrixsquarebound}
 Let $X$ be a finite dimensional Banach space, $B_1$, $B_2$ be two real-valued 
Wiener processes, $f_1, f_2:\mathbb R_+ \times \Omega \to X$ be two 
stochastically integrable functions. Let us define $M:= f_1 \cdot B_1 + f_2 
\cdot B_2$, $N:= f_2 \cdot B_1 - f_1 \cdot B_2$. Then for each $T\geq 0$
 \begin{equation*}
  (\mathbb E\|N_T\|^p)^{\frac 1p} \leq \beta_{p, X}^2(\mathbb E\|M_T\|^p)^{\frac 
1p}.
 \end{equation*}
\end{prop}
\begin{proof}
 The theorem follows from \cite{Y17MartDec}. Nevertheless we wish to illustrate 
an easier and more specific proof. Let $\widetilde B_1, \widetilde 
B_2:\widetilde {\Omega} \times \mathbb R_+ \to \mathbb R$ be two Wiener process 
defined on an enlarged probability space $(\widetilde{\Omega}, \widetilde 
{\mathcal F}, \widetilde {\mathbb P})$ with an enlarged filtration 
$\widetilde{\mathbb F} = (\widetilde{\mathcal F}_t)_{t\geq 0}$ such that 
$\widetilde B_1$ and $\widetilde B_2$ are independent of $\mathcal F$. Then 
by applying the decoupling theorem \cite[Theorem 4.4.1]{HNVW1} twice
(see also \cite{Mcc89}) and the fact that $-\widetilde B_1$ is a Wiener process
 \begin{align*}
  \mathbb E\|N_T\|^p = \mathbb E\|(f_2\cdot B_1)_T - (f_1\cdot B_2)_T\|^p &\leq 
\beta_{p,X}^p\mathbb E\|(f_2\cdot \widetilde B_1)_T - (f_1\cdot \widetilde 
B_2)_T\|^p\\
  &=\beta_{p,X}^p\mathbb E\|(f_1\cdot (-\widetilde B_2))_T + (f_2\cdot 
\widetilde B_1)_T\|^p\\
  &\leq \beta_{p,X}^{2p}\mathbb E\|(f_1\cdot  B_1)_T + (f_2\cdot  B_2)_T\|^p\\
  &=\beta_{p,X}^{2p}\mathbb E\|M_T\|^p.
 \end{align*}
\end{proof}

Let $p\in (1,\infty)$. A natural question is whether there exists a constant 
$C_p >0$ such that
\begin{equation}\label{eq:lindepend}
\|\mathcal H_X\|_{\mathcal L(L^p(\mathbb R; X))} \leq C_p \beta_{p,X}.
\end{equation}
Then the following theorem is applicable.
\begin{theorem}\label{thm:C_p}
 Let $X$ be a Banach space, $p\in(1,\infty)$. Then there exists $C_p \geq 1$ 
such that \eqref{eq:lindepend} holds if there exists \textbf{some} Burkholder 
function $U: X\times X \to \mathbb R$ such that $U$ is continuous and a.s.\ 
twice Fr\'echet differentiable, $U(x, y) \geq \|y\|^p - (C_p\beta_{p,X})^p 
\|x\|^p$ for any $x, y\in X$, $U(\alpha x, \alpha y) = |\alpha|^p U(x, y)$ for 
any $\alpha \in \mathbb R$ and $x, y\in X$, and the function
  \begin{equation*}
   t\mapsto U(x+tz_1, y+tz_2)+ U(x+ tz_2, y-tz_1),\;\;\; t\in \mathbb R,
  \end{equation*}
  or, equivalently,
  \begin{equation*}
   t\mapsto U(x+tz_1, y+tz_2)+ U(x- tz_2, y+tz_1),\;\;\; t\in \mathbb R,
  \end{equation*}
  is concave for each $x,y,z_1,z_2 \in X$ at $t=0$.
\end{theorem}

For the proof of Theorem \ref{thm:C_p} we will need a variant of the It\^o 
formula for a general basis of a finite dimensional linear space.

\begin{defi}\label{def:corrbasis}
 Let $d$ be a natural number, $E$ be a $d$-dimensional linear space, 
$(e_n)_{n=1}^d$ be a basis of~$E$. Then $(e_n^*)_{n=1}^d\subset E^*$ is called 
the {\it corresponding dual basis} of $(e_n)_{n=1}^d$ if $\langle e_n, 
e_m^*\rangle = \delta_{nm}$ for each $m,n=1,\ldots, d$.
\end{defi}

 Note that the corresponding dual basis is uniquely determined. Moreover, if 
$(e_n^*)_{n=1}^d$ is the corresponding dual basis of $(e_n)_{n=1}^d$, then, the 
other way around, $(e_n)_{n=1}^d$ is the corresponding dual basis of 
$(e_n^*)_{n=1}^d$ (here we identify $E^{**}$ with $E$ in the natural way).

The following theorem is a variation of \cite[Theorem 26.7]{Kal} which does not 
use the Hilbert space structure of a finite dimensional space.

\begin{theorem}[It\^o formula]\label{thm:itoformula}
 Let $d$ be a natural number, $X$ be a $d$-dimensional Banach space, $f\in 
C^2(X)$, $M:\mathbb R_+ \times \Omega \to X$ be a martingale. Let 
$(x_n)_{n=1}^d$ be a basis of $X$, $(x_n^*)_{n=1}^d$ be the corresponding dual 
basis. Then for each $t\geq 0$
 \begin{equation}\label{eq:itoformula}
  \begin{split}
   f(M_t) = f(M_0)&+ \int_0^t \langle \partial_xf(M_{s-}), \ud M_s\rangle\\
&+ \frac 12 \int_0^t \sum_{n,m=1}^d f_{x_n, x_m}(M_{s-})\ud[\langle M, 
x_n^*\rangle,\langle M, x_m^*\rangle]_s^c\\
&+ \sum_{s\leq t}(\Delta f(M_s) - \langle  \partial_xf(M_{s-}), \Delta 
M_{s}\rangle).
  \end{split}
 \end{equation}
\end{theorem}

\begin{proof}
 To apply \cite[Theorem 26.7]{Kal} one needs only to endow $X$ with a proper 
Euclidean norm $\vertiii{\cdot}$. Define $\vertiii{x} = \sum_{n=1}^d \langle x, 
x_n^*\rangle^2$ for each $x\in X$. Then $(x_n)_{n=1}^d$ is an orthonormal basis 
of $(X, \vertiii{\cdot})$, $M = \sum_{n=1}^d \langle M, x_n^*\rangle x_n$ is a 
decomposition of $M$ in this orthonormal basis, and therefore 
\eqref{eq:itoformula} is equivalent to the formula in \cite[Theorem 26.7]{Kal}.
\end{proof}

\begin{proof}[Proof of Theorem \ref{thm:C_p}]
 Let $M$ and $N$ be as in Proposition \ref{thm:01-10matrixsquarebound}. By the 
approximation argument we can suppose that  $M$ and $N$ have absolutely 
continuous distributions. Let $d$ be the dimension of $X$. Then by the It\^o 
formula in Theorem~\ref{thm:itoformula}
     \begin{equation}\label{eq:ineq+itoform}
      \begin{split}     
      \mathbb E\|N_t\|_X^p - (C_p\beta_{p, X})^p\mathbb E\|M_t\|_X^p &\leq 
\mathbb E U(M_t, N_t)
      = \mathbb EU(M_0, N_0) \\
      &+ \mathbb E\int_0^t\langle \partial_x U(M_{s}, N_{s}),\ud M_s\rangle\\
      &+ \mathbb E\int_0^t\langle \partial_y U(M_{s}, N_{s}),\ud N_s\rangle + 
\frac 12\mathbb E I,
      \end{split}
     \end{equation}
     where 
     \begin{multline}\label{eq:Iexpression}
            I = \int_0^t \sum_{i,j=1}^d(U_{x_i,x_j}(M_{s}, N_{s}) \ud [\langle 
x_i^*, M_s\rangle, \langle x_j^*, M_s\rangle]\\
            + 2U_{x_i,y_j}(M_{s}, N_{s}) \ud [\langle x_i^*, M_s\rangle, \langle 
y_j^*, N_s\rangle]\\
            + U_{y_i,y_j}(M_{s}, N_{s}) \ud [\langle y_i^*, N_s\rangle, \langle 
y_j^*, N_s\rangle]),
     \end{multline}
where $(x_i)_{i=1}^d = (y_i)_{i=1}^d\subset X$ is the same basis of $X$, and 
$(x_i^*)_{i=1}^d = (y_i^*)_{i=1}^d\subset X^*$ are the same corresponding dual 
bases of $X^*$. 

Notice that by Remark \ref{rem:propsofU} $\mathbb EU(M_0, N_0)\leq 0$ since 
$\|N_0\|\leq \|M_0\|$ a.s.\ and $C_p, \beta_{p, X}\geq 1$, and that
\[
 \mathbb E\Bigl(\int_0^t\langle \partial_x U(M_{s}, N_{s}),\ud M_s\rangle + 
\int_0^t\langle \partial_y U(M_{s}, N_{s}),\ud N_s\rangle \Bigr)= 0,
\]
since due to the same type of discussion as was done in the proof of 
Theorem~\ref{thm:weakdiffsubproc}, $\int_0^\cdot\langle \partial_x U(M_{s}, 
N_{s}),\ud M_s\rangle+\int_0^\cdot\langle \partial_y U(M_{s}, N_{s}),\ud 
N_s\rangle$ is a martingale which starts at zero.

Let us now prove that $I \leq 0$. For each $i = 1,2,\ldots, d$ we define $f_i^1 
:= \langle x^*_i, f_1\rangle$ and $f_i^2 := \langle x^*_i, f_2\rangle$. Then for 
each $i, j = 1,2\ldots, d$ one has that
\begin{equation}\label{eq:MMNN}
 \ud[\langle x_i^*, M_s\rangle, \langle x_j^*, M_s\rangle] = \ud[\langle y_i^*, 
N_s\rangle, \langle y_j^*, N_s\rangle] = (f_i^1f_j^1 + f_i^2f_j^2)\ud t, 
\end{equation}
and
\begin{equation}\label{eq:MN}
  \ud [\langle x_i^*, M_s\rangle, \langle y_j^*, N_s\rangle] = (f_i^1f_j^2 - 
f_i^2f_j^1)\ud t.
\end{equation}
Notice also that for each $x,y \in X$
\begin{equation}\label{eq:secondderivative}
 \begin{split}
  \frac{\partial^2}{\partial u^2} U(x + uf_1, y+uf_2)|_{u=0} &= 
\sum_{i,j=1}^d((U_{x_i^*,x_j^*}(x,y)f_i^1f_j^1 + 2U_{x_i^*,y_j^*}(x, 
y)f_i^1f_j^2\\ 
  &+ U_{y_i^*,y_j^*}(x, y) f_i^2f_j^2),\\
   \frac{\partial^2}{\partial u^2} U(x + uf_2, y-uf_1)|_{u=0} &=  
\frac{\partial^2}{\partial u^2} U(x - uf_2, y+uf_1)|_{u=0} 
\\=\sum_{i,j=1}^d((U_{x_i^*,x_j^*}(x,y)f_i^2f_j^2 &- 2U_{x_i^*,y_j^*}(x, 
y)f_i^2f_j^1
  + U_{y_i^*,y_j^*}(x, y) f_i^1f_j^1).
 \end{split}
\end{equation}
Therefore by \eqref{eq:Iexpression}, \eqref{eq:MMNN}, \eqref{eq:MN}, and 
\eqref{eq:secondderivative} we have that
\begin{align*}
I &= \int_0^t \sum_{i,j=1}^d((U_{x_i^*,x_j^*}(M_{s-}, N_{s-})(f_i^1f_j^1 + 
f_i^2f_j^2)+ 2U_{x_i^*,y_j^*}(M_{s-}, N_{s-})(f_i^1f_j^2 - f_i^2f_j^1) \\
&+ U_{y_i^*,y_j^*}(M_{s-}, N_{s-})(f_i^1f_j^1 + f_i^2f_j^2))\ud t = 
\int_0^t\frac{\partial^2}{\partial u^2} U(M_{s-} + uf_1, N_{s-}+uf_2)|_{u=0} \\ 
&+ \frac{\partial^2}{\partial u^2} U(M_{s-} + uf_2, N_{s-}-uf_1)|_{u=0} \ud s\\
&= \int_0^t\frac{\partial^2}{\partial u^2}\Bigl(U(M_{s-} + uf_1, N_{s-}+uf_2) 
+U(M_{s-} + uf_2, N_{s-}-uf_1)\Bigr)\Big|_{u=0}\ud s,
\end{align*}
and thanks to the concavity of $U(x + uf_1, y+uf_2) +U(x + uf_2, y-uf_1)$ in 
point $u=0$ for each $x,y \in X$ one deduces that a.s.\ $I\leq 0$. Then thanks 
to \eqref{eq:ineq+itoform} one has that
\begin{equation}\label{eq:N_tM_tC_pbeta_p,X}
 \mathbb E\|N_t\|_X^p - (C_p\beta_{p, X})^p\mathbb E\|M_t\|_X^p\leq \mathbb E 
U(M_t, N_t) \leq 0.
\end{equation}

Now one can prove that \eqref{eq:N_tM_tC_pbeta_p,X} implies \eqref{eq:lindepend} 
in the same way as it was done for instance in \cite[Theorem 3]{Gar85}, 
\cite[p.592]{BanWang95} or \cite[Chapter 3]{Burk01}.
\end{proof}

\begin{remark}\label{rem:C_p}
 Note that if $X$ is a finite dimensional Hilbert space, then one gets condition 
(iii) in Theorem \ref{thm:C_p} for free from \cite{Wang}. Indeed, let $U:X 
\times X \to \mathbb R$ be as in \cite[p.\,527]{Wang}, namely
 \[
U(x, y) = p(1-1/p^*)^{p-1}(\|y\|-(p^*-1)\|x\|)(\|x\|+\|y\|)^{p-1},\;\;\; x, y\in 
X.  
 \]
Then $U$ is a.s.\ twice Fr\'echet differentiable, and thanks to the property (c) 
of $U$, which is given on \cite[p.\,527]{Wang}, for all nonzero $x, y\in X$ 
there exists a constant $c(x, y)\geq 0$ such that
 \begin{align*}
   \langle \partial_{xx} U(x,y),(h,h)\rangle &+ 2\langle \partial_{xy} 
U(x,y),(h,k)\rangle+\langle \partial_{yy}U(x,y),(k,k)\rangle\\
   &\leq -c(x,y)(\|h\|^2-\|k\|^2),\;\;\; h,k \in X.
 \end{align*}
Therefore for any $z_1, z_2 \in X$
\begin{align*}
 \frac{\partial^2}{\partial t^2} \bigl[U(x&+tz_1, y+tz_2)+ U(x+ tz_2, 
y-tz_1)\bigr]\Big|_{t=0}\\
 &= \langle \partial_{xx} U(x,y),(z_1,z_1)\rangle + 2\langle  \partial_{xy} 
U(x,y),(z_1,z_2)\rangle+\langle \partial_{yy}U(x,y),(z_2,z_2)\rangle\\
 &\quad+\langle \partial_{xx} U(x,y),(z_2,z_2)\rangle - 2\langle  \partial_{xy} 
U(x,y),(z_2,z_1)\rangle+\langle \partial_{yy}U(x,y),(z_1,z_1)\rangle\\\
 &\leq -c(x,y)(\|z_1\|^2-\|z_2\|^2)-c(x,y)(\|z_2\|^2-\|z_1\|^2)= 0.
\end{align*}
\end{remark}

\subsection{General conjecture}

By Theorem \ref{thm:C_p} the estimate \eqref{eq:lindepend} is a direct 
corollary of the following conjecture.
\begin{conj}\label{thm:main}                                           
 Let $X$ be a finite dimensional Banach space, $p\in (1,\infty)$. Then there 
exists $C_p\geq 1$ such that for each pair of continuous martingales 
$M,N:\mathbb R_+\times \Omega \to X$ such that $N$ is weakly differentially 
subordinated to $M$ one has that for each $t\geq 0$
 \begin{equation}\label{eq:main}
  (\mathbb E \|N_t\|^p)^{\frac 1p}\leq C_p\beta_{p,X}(\mathbb E 
\|M_t\|^p)^{\frac 1p}.
 \end{equation}
\end{conj}

\begin{remark}
 Notice that in \cite{Y17MartDec} the estimate \eqref{eq:main} is proven with the constant $\beta_{p, X}^2$ instead of $C_p\beta_{p,X}$. Moreover, it is shown in \cite{Y17MartDec} that $C_p$ can not be less then 1.
\end{remark}

We wish to finish by pointing out some particular cases in which 
Conjecture~\ref{thm:main} holds. These results are about stochastic integration 
with respect to a Wiener process. Recall that we assume that $X$ is a finite 
dimensional space. Later we will need a couple of definitions.

Let $W^H:\mathbb R_+ \times H \to L^2(\Omega)$ be an {\it $H$-cylindrical Brownian motion}, i.e.
\begin{itemize}
 \item $(W^Hh_1, \ldots, W^Hh_d):\mathbb R_+ \times \Omega \to\mathbb R^d$ is a 
$d$-dimensional Wiener process for all $d\geq 1$ and  $h_1,\ldots, h_d \in H$,
 \item $\mathbb E W^H(t)h\,W^H(s)g = \langle h,g\rangle \min\{t,s\}$
 $\forall h,g \in H$, $t,s \geq 0$.
\end{itemize}
(We refer the reader to \cite[Chapter~4.1]{DPZ} for further details). Let $X$ be a Banach space, $\Phi:\mathbb R_+ \times \Omega \to 
\mathcal L(H,X)$ be elementary progressive of the form \eqref{eq:elprog}. Then 
we define a {\it stochastic integral} $\Phi \cdot W^H:\mathbb R_+ \times \Omega 
\to X$ of $\Phi$ with respect to $W^H$ in the following way:
\begin{equation*}
 (\Phi \cdot W^H)_t = \sum_{k=1}^K\sum_{m=1}^M \mathbf 1_{B_{mk}}
\sum_{n=1}^N (W^H(t_k\wedge t)h_n- W^H(t_{k-1}\wedge t)h_n)x_{kmn},\;\; t\geq 0.
\end{equation*}

The following lemma is a multidimensional variant of \cite[(3.2.19)]{KS} and it 
is closely connected with Lemma \ref{lemma:KunWat}.
\begin{lemma}\label{lemma:covstochintwrtcylbrmo}
 Let $X = \mathbb R$, $\Phi, \Psi:\mathbb R_+ \times \Omega \to \mathcal L(H, 
\mathbb R)$ be elementary progressive. Then for all $t\geq 0$ a.s.
 \[
  [\Phi \cdot W^H, \Psi \cdot W^H]_t = \int_0^t \langle \Phi^*(s), \Psi^*(s) 
\rangle\ud s.
 \]
\end{lemma}

The reader can find more on stochastic integration with respect to an 
$H$-cylindrical Brownian motion in the UMD case in \cite{NVW}.

\begin{theorem}\label{thm:selfadjmatrixsquarebound}
 Let $X$ be a finite dimensional Banach space, $W^H$ be an $H$-cylindrical 
Brownian motion, $\Phi:\mathbb R_+ \times \Omega \to \mathcal L(H,X)$ be 
stochastically integrable with respect to $W^H$ function. Let $A\in \mathcal 
L(H)$ be self-adjoint. Then
 \begin{equation}\label{eq:selfadjmatrixsquarebound}
  (\mathbb E\|((\Phi A)  \cdot W^H)_{\infty}\|_X^p)^{\frac 1p} \leq \beta_{p, 
X}\|A\|(\mathbb E\|(\Phi \cdot W^H)_{\infty}\|_X^p)^{\frac 1p}.
 \end{equation}
\end{theorem}

 Notice that by Lemma \ref{lemma:covstochintwrtcylbrmo} for each $x^* \in 
X^*$ and $ 0\leq s<t<\infty$ a.s.
\begin{align*}
 [\langle (\Phi A)  \cdot W^H, x^*\rangle]_t -  [\langle (\Phi A)  \cdot W^H, 
x^*\rangle]_s  &= \int_s^t \|A \Phi^*(r) x^*\|^2\ud r\\
  &\leq \|A\|^2\int_s^t \|\Phi^*(r) x^*\|^2\ud r\\
 &= \|A\|^2\bigl([\langle \Phi  \cdot W^H, x^*\rangle]_t-  [\langle \Phi   \cdot 
W^H, x^*\rangle]_s\bigr).
\end{align*}
Hence if $\|A\|\leq 1$, then $(\Phi A)  \cdot W^H$ is weakly differentially 
subordinated to $\Phi  \cdot W^H$, and therefore Theorem 
\ref{thm:selfadjmatrixsquarebound} provides us with a special case of Conjecture 
\ref{thm:main}.

\begin{proof}[Proof of Theorem \ref{thm:selfadjmatrixsquarebound}]
Due to \cite[Theorem 3.6]{NVW} it is enough to show 
\eqref{eq:selfadjmatrixsquarebound} for elementary progressive process $\Phi$. Let 
$(h_n)_{n\geq 1}$ be an orthogonal basis of $H$, and let $\Phi$ be of the form 
\eqref{eq:elprog}. For each $n\geq 1$ we define $P_n\in \mathcal L(H)$ as an 
orthonormal projection onto $\text{span}(h_1,\ldots,h_n)$, and set $A_n:= P_n A 
P_n$. Notice that $\|A_n\|\leq \|A\|$. Then
\begin{equation*}
 \|((\Phi A - \Phi A_n)  \cdot W^H)_{\infty}\|_{L^p(\Omega; X)} \to 0,\;\;\;n\to 
\infty,
\end{equation*}
 so it is sufficient to prove \eqref{eq:selfadjmatrixsquarebound} for $A$ with a 
finite dimensional range, and we can suppose that there exists $d\geq 1$ such 
that $\text{ran }A\subset\text{span}(h_1,\ldots,h_d)$. This implies that $A$ is 
compact self-adjoint, so we can change the first $d$ vectors $h_1,\ldots,h_d$ of 
the orthonormal basis in such a way that $A = \sum_{n=1}^d \lambda_n h_n\otimes 
h_n$, where $(\lambda_n)_{n\geq 1}$ is a real-valued sequence. Without loss of generality we can assume that $|\lambda_1|\geq \cdots \geq |\lambda_d|$ and $|\lambda_1| = \|A\|$. Notice that under this change of coordinates  
$\Phi$ remains elementary progressive (perhaps of a different form).  
Therefore by the martingale transform theorem \cite[Theorem~4.2.25]{HNVW1}:
\begin{align*}
 \mathbb E\|((\Phi A)  \cdot W^H)_{\infty}\|_X^p &= \mathbb 
E\|\sum_{n=1}^d((\Phi Ah_n)  \cdot W^H(h_n))_{\infty}\|_X^p\\
 &=\|A\|^p\mathbb E\Bigl\|\sum_{n=1}^d\frac{\lambda_n}{\|A\|}((\Phi h_n)  \cdot 
W^H(h_n))_{\infty}\Bigr\|_X^p\\
 &\leq \beta_{p,X}^p\|A\|^p\mathbb E\Bigl\|\sum_{n=1}^N((\Phi h_n)  \cdot 
W^H(h_n))_{\infty}\Bigr\|_X^p.
\end{align*}
The last inequality holds because of structure of $\Phi$ so that one can rewrite 
$(\Phi h_n)  \cdot W^H(h_n)$ as a summation in time, and because 
$(W^H(h_n))_{n\geq 1}$ is a sequence of independent Wiener processes.
\end{proof}

\begin{remark}
 Theorem \ref{thm:selfadjmatrixsquarebound} in fact can be shown using 
\cite[Proposition 3.7.(i)]{GM-SS}.
\end{remark}

\begin{remark}
An analogue of Theorem \ref{thm:selfadjmatrixsquarebound} for antisymmetric $A$ 
(i.e.\ $A$ such that $A^* = -A$) remains open. It is important for instance for 
the possible estimate \eqref{eq:lindepend}. Indeed, in Proposition 
\ref{thm:01-10matrixsquarebound} the Hilbert space $H$ can be taken 
2-dimensional, $A = \bigl(\begin{smallmatrix}
0 & -1\\1 & 0 \end{smallmatrix}\bigr)$, and $\Phi :\mathbb R_+ \times \Omega \to 
\mathcal L(H, X)$ is such that $\Phi \bigl(\begin{smallmatrix}
a\\b \end{smallmatrix}\bigr) = af_1 + bf_2$ for each $a, b\in \mathbb R$. Then 
$M = \Phi \cdot W^H$, $N = (\Phi A)\cdot W^H$, and if one shows 
\eqref{eq:selfadjmatrixsquarebound} for an antisymmetric operator $A$, then one 
automatically gains~\eqref{eq:lindepend}.
\end{remark}

The next theorem shows that Conjecture \ref{thm:main} holds for stochastic 
integrals with respect to a one-dimensional Wiener process.

\begin{theorem}\label{thm:onedimint}
 Let $X$ be a finite dimensional Banach space, $W:\mathbb R_+ \times \Omega \to 
\mathbb R$ be a one-dimensional Wiener process, $\Phi, \Psi:\mathbb R_+ \times 
\Omega \to X$ be stochastically integrable with respect to $W$, $M=\Phi \cdot 
W$, $N=\Psi\cdot W$. Let $N$ be weakly differentially subordinated to $M$. Then 
for each $p\in (1,\infty)$,
 \begin{equation}\label{eq:onedimint}
  \mathbb E\|N_{\infty}\|^p\leq \beta_{p,X}^p\mathbb E\|M_{\infty}\|^p.
 \end{equation}
\end{theorem}

\begin{proof}
 Without loss of generality suppose that there exists $T\geq 0$ such that 
$\Phi\mathbf 1_{[T,\infty]} = \Psi\mathbf 1_{[T,\infty]} = 0$. Since $N$ is 
weakly differentially subordinated to $M$, by the It\^o isomorphism for each 
$x^* \in X^*$, $0\leq s<t<\infty$ we have~a.s.
 \begin{align*}
    [\langle x^*,N\rangle]_t - [\langle x^*,N\rangle]_s &= \int_{s}^t|\langle 
x^*, \Psi(r)\rangle|^2\ud r\\
    &\leq\int_{s}^t|\langle x^*, \Phi(r)\rangle|^2\ud r =  [\langle 
x^*,M\rangle]_t - [\langle x^*,M\rangle]_s.
 \end{align*}
Therefore we can deduce that $|\langle x^*, \Psi\rangle| \leq |\langle x^*, 
\Phi\rangle|$ a.s.\ on $\mathbb R_+ \times \Omega$. By 
Lemma~\ref{lemma:ajustformeasfunc} there exists progressively measurable 
$a:\mathbb R_+ \times \Omega \to \mathbb R$ such that $|a|\leq 1$ on $\mathbb 
R_+ \times \Omega$ and $\Psi = a\Phi$ a.s.\ on $\mathbb R_+ \times \Omega$. Now 
for each $n\geq 1$ set $a_n:\mathbb R_+ \times \Omega \to \mathbb R$, 
$\Phi_n:\mathbb R_+ \times \Omega \to X$ be elementary progressively measurable 
such that $|a_n|\leq 1$, $a_n \to a$ a.s.\ on $\mathbb R_+ \times \Omega$ and 
$\mathbb E\int_0^T \|\Phi(t)-\Phi_n(t)\|^2\ud t\to 0$ as $n\to\infty$. Then by 
the triangle inequality
\begin{equation}\label{eq:Psi-a_nPhi_n}
 \begin{split}
  \Bigl(\mathbb E\int_0^T \|\Psi(t)-a_n(t)\Phi_n(t)\|^2\ud t\Bigr)^{\frac 
12}&\leq  \Bigl(\mathbb E\int_0^T \|\Phi(t)\|^2(a(t)-a_n(t))^2\ud t\Bigr)^{\frac 
12}\\ 
  &\quad+  \Bigl(\mathbb E\int_0^T \|\Phi(t)-\Phi_n(t)\|^2a_n^2\ud 
t\Bigr)^{\frac 12},
 \end{split}
\end{equation}
which vanishes as $n\to \infty$ by the dominated convergence theorem. For each 
$n\geq 1$ the inequality
\[
  (\mathbb E\|((a_n\Phi_n) \cdot W)_{\infty}\|^p)^{\frac 1p}\leq 
\beta_{p,X}(\mathbb E\|(\Phi_n \cdot W)_{\infty}\|^p)^{\frac 1p}
\]
holds thanks to the martingale transform theorem \cite[Theorem 4.2.25]{HNVW1}. 
Then~\eqref{eq:onedimint} follows from the previous estimate and 
\eqref{eq:Psi-a_nPhi_n} when one lets $n$ go to infinity.
\end{proof}
\begin{remark}
 Let $W$ be a one-dimensional Wiener process, $\mathbb F$ be a filtration which 
is generated by $W$. Let $M,N:\mathbb R_+ \times \Omega \to X$ be $\mathbb 
F$-martingales such that $M_0=N_0=0$ and $N$ is weakly differentially 
subordinated to $M$. Then thanks to the It\^o isomorphism \cite[Theorem 
3.5]{NVW} there exist progressively measurable $\Phi, \Psi:\mathbb R_+ \times 
\Omega \to X$ such that $M = \Phi \cdot W$and $N=\Psi \cdot W$, and thanks to 
Theorem \ref{thm:onedimint}
 \[
  \mathbb E\|N_{\infty}\|^p \leq \beta_{p,X}^p\mathbb E \|M_{\infty}\|^p,\;\;\; 
p\in(1,\infty).
 \]
 This shows that on certain probability spaces the estimate \eqref{eq:main} 
automatically holds with a constant $C_p=1$.
\end{remark}

\bibliographystyle{plain}

\def\cprime{$'$} \def\polhk#1{\setbox0=\hbox{#1}{\ooalign{\hidewidth
  \lower1.5ex\hbox{`}\hidewidth\crcr\unhbox0}}}
  \def\polhk#1{\setbox0=\hbox{#1}{\ooalign{\hidewidth
  \lower1.5ex\hbox{`}\hidewidth\crcr\unhbox0}}} \def\cprime{$'$}

\end{document}